\documentclass[letterpaper,11pt]{article}

\usepackage[Arxiv]{optional}

\usepackage[dvipsnames]{xcolor}

\usepackage{tikz-cd}
\usepackage{tikz}
\usepackage{adjustbox}
\usepackage{afterpage}
\usepackage{enumerate}
\usepackage[shortlabels]{enumitem}
\usepackage[linesnumbered,algoruled,boxed,lined,noend]{algorithm2e}
\usepackage[normalem]{ulem}
\usepackage{multirow}
\usepackage{amsmath, 
            amsthm,
            amssymb,
            mathtools,
            thmtools}
\usepackage{extarrows}
\usepackage{hyperref}
\usepackage{float}
\usepackage[all]{xy}
\usepackage{array}
\usepackage[mathlines]{lineno}

\usepackage[margin=1in]{geometry}

\theoremstyle{plain} 
\newtheorem{theorem}{Theorem}[section]
\newtheorem{lemma}[theorem]{Lemma}
\newtheorem{proposition}[theorem]{Proposition}
\newtheorem{corollary}[theorem]{Corollary}

\theoremstyle{definition}
\newtheorem{definition}[theorem]{Definition}

\usepackage{amssymb}

\def\refeq#1{\if\workingver y(\ref{#1})-[[#1]]\else(\ref{#1})\fi}
\def\refth#1{\if\workingver y\ref{#1}-[[#1]]\else\ref{#1}\fi}
\def\mylabel#1{\if\workingver y\label{#1}{\bf\ \ [[#1]]\ \ }\else\label{#1}\fi}
\def\mybibitem#1{\if\workingver y\bibitem{#1}{\bf\ \ [[#1]]\ \
}\else\bibitem{#1}\fi}

\renewcommand{\emptyset}{\varnothing}
\renewcommand{\rho}{\varrho}
\renewcommand{\phi}{\varphi}
\renewcommand{\epsilon}{\varepsilon}

\def\cA{\text{$\mathcal A$}}
\def\cB{\text{$\mathcal B$}}

\def\cV{\text{$\mathcal V$}}

\def\cX{\text{$\mathcal X$}}

\newcommand{\id}{\operatorname{id}}
\newcommand{\cl}{\operatorname{cl}}

\newcommand{\im}{\operatorname{im}}
\newcommand{\coker}{\operatorname{coker}}
\newcommand{\coim}{\operatorname{coim}}

\renewcommand{\emptyset}{\varnothing}

\def\begeq#1{\begin{equation}\mylabel{#1}}
\def\endeq{\end{equation}}

\def\mathobj#1{\mbox{$#1$}}

\def\PP{\mathobj{\mathbb{P}}}

\def\ZZ{\mathobj{\mathbb{Z}}}

\def\implies{\;\Rightarrow\;}

\def\0#1{\hbox{\kern25pt}$ #1 $\\}
\def\1#1{\hbox{\kern40pt}$ #1 $\\}
\def\2#1{\hbox{\kern55pt}$ #1 $\\}
\def\3#1{\hbox{\kern70pt}$ #1 $\\}

\newcounter{li}

\def\begalg#1{\begin{algo}\mylabel{#1}\normalshape:\small\baselineskip 10pt\\}
\def\endalg{\end{algo}}

\def\Figures(include=#1,cat=#2){
  \renewcommand{\textfraction}{.20}
  \renewcommand{\topfraction}{.80}
  \renewcommand{\bottomfraction}{.80}
  \renewcommand{\floatpagefraction}{.80}
  \newcount\figcount
  \figcount=0
  \let\includefigures=#1
  \def\figcat{#2}
}

\def\FigureFromFile[#1][#2](#3)#4
{
  \begin{figure}[htbp]
     \global\advance\figcount by 1
     \if\includefigures y\special{anisoscale #1.wmf, \the\hsize #2}\fi
     \vspace{#2}
     \caption{#4}
     \mylabel{#3}
   \end{figure}
}

\def\FigureFromFileTwoD[#1][#2,#3](#4)#5
{
  \begin{figure}[htbp]
     \global\advance\figcount by 1
     \if\includefigures y\special{anisoscale #1.wmf, #2 #3}\fi
     \vspace{#2}
     \caption{#5}
     \mylabel{#4}
   \end{figure}
}

\def\FigureF<#1>[#2](#3)#4
{
  \begin{figure}[htbp]
     \global\advance\figcount by 1
     \if\includefigures y\special{anisoscale \figcat/fig\number\figcount.wmf,
       \the\hsize #2}
     \fi
     \if\includefigures p
       \leavevmode
       \epsfxsize=\hsize
       \epsffile{#1}
     \fi
     \if\includefigures y
          \vspace{#2}
     \fi
     \caption{#4}
     \mylabel{#3}
   \end{figure}
}

\def\Figure[#1](#2)#3
{
  \begin{figure}[htbp]
     \global\advance\figcount by 1
     \if\includefigures y\special{anisoscale \figcat/fig\number\figcount.wmf,
       \the\hsize #1}
     \fi
     \if\includefigures p
       \leavevmode
       \epsfxsize=\hsize
       \epsffile{fig\number\figcount.eps}
     \fi
     \if\includefigures y
          \vspace{#1}
     \fi
     \caption{#3}
     \mylabel{#2}
   \end{figure}
}

\DeclareMathOperator{\mo}{mo}

\newcommand{\inscr}{\sqsubseteq}
\newcommand{\ovscr}{\sqsupseteq}
\newcommand{\inovscr}{\Longleftrightarrow}
\newcommand{\xinovscr}[1]{\xLongleftrightarrow{\idxfb{t}}}

\newcommand{\inc}{\kappa}
\newcommand{\facerel}{\prec}
\newcommand{\Zgroup}{\mathsf{Z}}
\newcommand{\Bgroup}{\mathsf{B}}
\newcommand{\Cgroup}{\mathsf{C}}

\newcommand{\FV}{F_\mathcal{V}}
\newcommand{\GV}{G_\mathcal{V}}

\newcommand{\BD}{\cB}
\newcommand{\zzBD}{\mathfrak{B}}
\newcommand{\zzV}{\mathfrak{V}}

\newcommand{\lmap}[1]{\iota_{#1}}
\newcommand{\lmapx}[1]{\iota_{#1\ast}}
\newcommand{\imap}[1]{j_{#1}}
\newcommand{\imapx}[1]{j_{#1\ast}}
\newcommand{\projmap}[1]{\pi_{#1}}
\newcommand{\projmapx}[1]{\pi_{#1\ast}}

\newcommand{\idxmap}[1]{\iota_{#1}}
\newcommand{\idxfb}[1]{\overleftrightarrow{\iota}_{\!\!#1}}
\newcommand{\idxfwd}[1]{{\iota}_{#1}}
\newcommand{\idxbck}[1]{\overleftarrow{\iota}_{\!\!#1}}

\newcommand{\VecSp}{\mathbf{Vec}}
\newcommand{\field}{\Bbbk}

\definecolor{yellow}{RGB}{255,225,55}
\newcommand{\tamal}[1]{\textcolor{blue}{#1}}

\newcommand\michal[1]{\textcolor{ForestGreen}{[ML: #1]}}

\newcommand{\low}{{\tt piv}}

\newcommand{\md}{\mathcal{B}}

\newcommand{\bl}{B}

\definecolor{dark-gray}{RGB}{64,64,64}
\definecolor{medium-gray}{RGB}{114,114,114}
\definecolor{light-gray}{RGB}{190,190,190}

\newcommand{\cancel}[1]

\title{Computing Conley-Morse Persistence Barcode Efficiently by Updating Matrix Decompositions}

\author{Tamal K. Dey \thanks{Department of Computer Science, Purdue University, West Lafayette, Indiana, USA, \texttt{tamaldey@purdue.edu}}
 \and 
    Andrew Haas \thanks{Department of Computer Science, Purdue University, West Lafayette, Indiana, USA, \texttt{haas60@purdue.edu}}
\and
    Micha\l{} Lipi\'nski \thanks{Institute of Science and Technology Austria (ISTA), Klosterneuburg, Austria, \texttt{michal.lipinski@ist.ac.at}}}

\begin{document}

\maketitle

\begin{abstract}
Recent advances in combinatorial dynamical systems that generalize the classic
discrete Morse theory have prompted algorithmic studies of combinatorial vector fields. In this regard, authors in~\cite{CMbarcodes2025} recently proposed the concept of
Conley-Morse persistence barcode that summarizes the continuation of invariant sets in an evolving vector field through homological persistence. 
They proposed an algorithm to compute this barcode using a filtration of the so called \emph{index pairs} on a poset called \emph{transition diagram}. 
The algorithm becomes costly due to multiple runs of zigzag persistence it executes on filtrations of `unwieldy' structures of index pairs. We overcome this difficulty by replacing the index pairs with \emph{blocks},
    which are structurally much simpler. 
These replacements need reversal of certain relations in the
transition diagram resulting in a much simpler algorithm.
The algorithm works by updating matrix decompositions
akin to computing `vineyard' in standard persistence.


\end{abstract}

\newpage

\section{Introduction}

The study of evolving dynamical systems recently appeared in topological data analysis (TDA). This includes work on sampled dynamics~\cite{GuMuKh2022,Kim:2021wx, TyMuKh2020}, time-varying gradient systems
in the language of multiparameter persistence~\cite{Bubenik2024}, and tracking the evolution
of critical points~\cite{King2017, MeLiCh2026, ReininghausHotz2012, Hotz2023}. 
Recent multivector field theory~\cite{LKMW2022, Mrozek2017} generalizing Forman's idea~\cite{Forman1998b} makes it possible to deal with non-gradient systems and offers an elegant  combinatorial framework for continuation theory~\cite{Franzosa1988}.
Based on that, the authors of~\cite{CMbarcodes2025} introduced Conley-Morse persistence barcode which allows for a systematic study of parameterized multivector fields through the lens of Conley index~\cite{Conley1978} and persistent homology~\cite{DW22, EdelsHarer2008}, thus linking bifurcation theory with TDA.


A parameterized combinatorial multivector field on a simplicial complex $X$ is a sequence of multivector fields $\zzV:\cV_1,\cV_2,\ldots,\cV_T$.  
To understand the evolution of the dynamics given by such a sequence, one can track isolated invariant sets~\cite{DeLiMrSl2022,CMbarcodes2025}, e.g., equilibria or periodic orbits generalizing the role of critical cells in Morse theory, 
    and their \emph{Conley indices}~\cite{Conley1978, MischMro_Conley_2002}, a homological invariant generalizing the Morse index.
The total dynamics of $\cV_i$ is characterized by a \emph{block partition} $\cB_i$ consisting of \emph{blocks}.
Each block encapsulates an isolated invariant set, thus, it can also be characterized by the Conley index.

The Conley index is defined as the relative homology of the so called \emph{index pair}. 
In~\cite{DeLiMrSl2022}, the authors noticed that changes in a Conley index across parameterization $\zzV$ can be expressed in the language of zigzag persistence~\cite{CaSiMo2009-zigzag}, through zigzag filtration of index pairs.
In~\cite{CMbarcodes2025}, the authors showed that Conley indices of all blocks in $\cB_i$ can be tracked simultaneously, which additionally allows one to study interactions between them.
This is achieved by constructing a \emph{transition diagram}, a filtration of index pairs over a poset $\bar{\PP}$, 
    which, in turn, induces \emph{Conley-Morse persistence module}, also over the poset $\overline{\PP}$. 
Although the poset $\overline{\PP}$ supporting this persistence module is not an $A_n$ type quiver (zigzag path),
    the supports of the indecomposables are zigzag paths forming \emph{Conley-Morse persistence barcode}. 
Essentially, the barcode represents the evolution of the Conley indices over $\zzV$.

The algorithm presented in~\cite{CMbarcodes2025} for the computation of the barcode runs incrementally over the parameter value while executing multiple runs of zigzag persistence algorithms~\cite{DW22,milosavljevic2011zigzag} on zigzag filtrations of `unwieldy' index pairs. 
Naturally, the algorithm becomes costly. 

In this paper, we show that a much simpler algorithm can compute the same Conley-Morse
persistence barcode. 
The key idea is to construct an alternative transition diagram over a poset $\PP$ based on the inclusion of blocks instead of index pairs.
However, these inclusions do not necessarily induce homomorphisms in homology.
As a remedy, we construct them \emph{indirectly} by `reversing' maps induced by inclusions for index pairs;
    though our algorithm does it implicitly. 
We show that the new persistence module supported over
    the modified transition diagram still provides the same Conley-Morse persistence barcode (Proposition~\ref{prop:CMmodule-simplified}).
The main upshot of this observation is that inclusions among the blocks at the
chain level allow us to design an algorithm that operates on updating matrix
decompositions akin to the well known `vineyard' computations for standard persistence~\cite{CohEdeMor2006}. 
Essentially, the algorithm maintains a $R_t=D_tU_t$ decomposition for each block boundary matrix $D_t$ for the vector field $\cV_t$ and, instead of
computing the decomposition afresh, it updates the current decomposition to
the next one at $t+1$ while tracking the bars in the barcode.

\section{Preliminaries}
A pair $(X, \kappa)$ is called a \emph{Lefschetz complex over a fixed field $\Bbbk$}, 
    where $X=\{X_q\}_{q\in{\mathbb{Z}}}$ is a finite set with~$\ZZ$ gradation, 
    and $\kappa:X\times X\rightarrow\Bbbk$, called \emph{incidence coefficient}, satisfies:
    $\inc(\sigma, \tau) \neq 0$ implies $\sigma\in X_q$ and $\tau\in X_{q-1}$ where $q\in\ZZ$; 
    and, for any $\sigma,\tau\in X$ we have 
    $\sum_{\mu\in X} \kappa(\sigma,\mu)\kappa(\mu,\tau) = 0$.
We call the elements of $X$ \emph{cells}.
In particular, any simplicial complex is a Lefschetz complex.
The incidence coefficient induces a partial order on cells -- the face relation -- denoted $\facerel$. 
A~subset $A\subset X$ is \emph{locally closed} if $\sigma,\tau\in A$, $\mu\in X$, $\sigma\facerel \mu\facerel \tau$ imply $\mu\in A$. 
Similarly, $A$ is a \emph{closed} in $X$ if $\tau\in A$, $\sigma\in X$, $\sigma\facerel \tau$ imply $\sigma\in A$.
The \emph{mouth} of a set $A$ is defined as $\mo A\coloneqq\cl A\setminus A$.
Equivalently, $A$ is locally closed if and only if $\mo A$ is closed.

A subset $A\subset X$ is a \emph{Lefschetz subcomplex} of $X$ if $\kappa$ restricted to $A$ still satisfies the definition. 
In particular, $A$ is a Lefschetz subcomplex if and only if it is locally closed in $X$ \cite[Theorem~3.1]{MroBat2009}.

Lefschetz complex induces a chain complex $(\Cgroup(X),\partial)$, where each $\Cgroup_q(X)$ is spanned by elements of $X_q$, 
    and the differential $\partial_q:\Cgroup_q(X) \rightarrow \Cgroup_{q-1}(X)$ is defined on basis elements $\sigma\in \Cgroup_q(X)$ as
    \begin{equation}\label{eq:boundary_homomorphism}
        \partial_q(\sigma) \coloneqq \sum_{\tau\in X} \inc(\sigma, \tau) \tau.
    \end{equation}
The subgroups of cycles and boundaries are defined in the standard way, that is $\Zgroup_q(X)\coloneqq\ker\partial_q$ and $\Bgroup_q(X)\coloneqq\im\partial_{q+1}$, respectively.
Then, the \emph{Lefschetz homology} of $X$ is $H_q(X)\coloneqq \Zgroup_q(X)/\Bgroup_q(X)$
and we write $H(X)\coloneqq\bigoplus_{q\in\mathbb{Z}}H_q(X)$.

Let $A$ be a closed subcomplex in $X$.
Then, the relative chain group $\Cgroup_q(X,A)$ consists of equivalence classes of form $c+\Cgroup_q(A)$, where $c\in\Cgroup_q(X)$.
Boundary homomorphism
    induces     $\partial_q^{(X,A)}:\Cgroup_q(X,A)\rightarrow\Cgroup_{q-1}(X,A)$, 
        and therefore we have 
        $\Zgroup_q(X,A)\coloneqq\ker\partial_q^{(X,A)}$, 
        $\Bgroup_q(X,A)\coloneqq\im\partial_{q+1}^{(X,A)}$, 
        and $H_q(X,A)\coloneqq \Zgroup_q(X,A)/\Bgroup_q(X,A)$ for relative cycles, boundaries and Lefschetz homology.
We skip the superscript in $\partial_q^{(X,A)}$ when it is clear that it is a relative boundary homomorphism.
As in the regular case, we write $H(X,A)\coloneqq\bigoplus_{q\in\mathbb{Z}}H_q(X,A)$.

If $X$ is a simplicial complex, the simplicial and the Lefschetz homology coincide. 
However, Lefschetz homology is particularly useful when dealing with locally closed sets.
\begin{proposition}\cite[Proposition~3.5.8]{MroWan2025}\label{prop:lefschetz_relative_homology}
    Let $A$ be a locally closed subset of a Lefschetz complex~$X$.
    Let $D$, $E$ be closed sets such that $E\subset D$ and $D\setminus E=A$, then
    we have an isomorphism $H(A)\cong H(D, E)$ in Lefschetz homology.
    \label{prop:lefschetzhom}
\end{proposition}


\subsection{Combinatorial Multivector Fields}
    \label{subsec:mvf}
    
\begin{figure}
    \centering
    \includegraphics[width=1.0\linewidth]{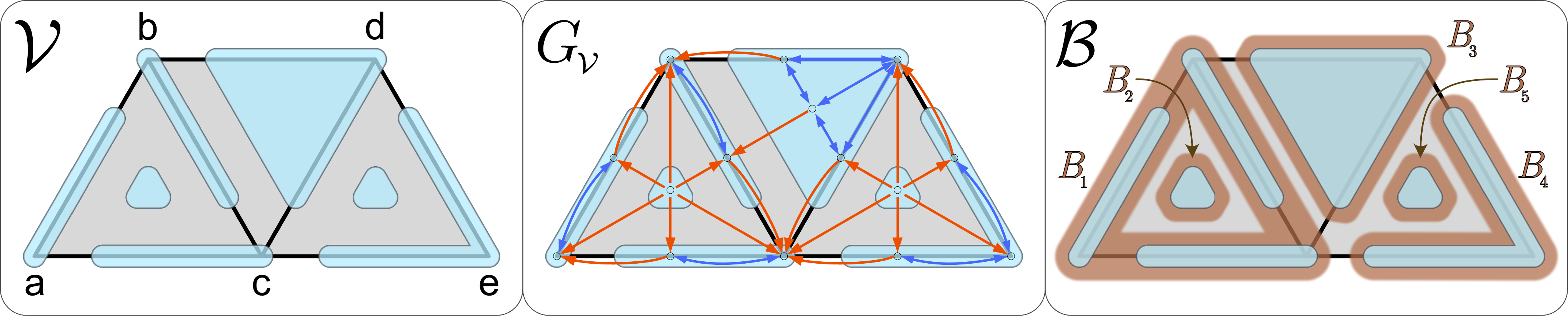}
    \caption{
    Left: a multivector field $\cV$ on a simplicial complex. 
    Right: the corresponding graph $G_{\cV}$.
    Blue arrows represent edges within multivectors, while red ones are across the multivectors. Self loops are omitted for clarity.
    Bottom: minimal block partition of $\cV$ highlighted with brown color.
    }
    \label{fig:mvf-example}
\end{figure}

A partition $\cV$ of a Lefschetz complex $X$ into locally closed subsets is called a \emph{multivector field}. We refer to elements of $\cV$ \emph{multivectors}.
Each cell $\sigma\in X$ is contained in exactly one multivector $V\in\cV$, which we denote by $[\sigma]_{\cV}$.
A multivector field $\cV$ induces a multivalued map $\FV(\sigma)\coloneqq\cl \sigma\cup [\sigma]_\cV$.
The map~$\FV$ can be interpreted as a directed graph $\GV$ on $X$ with directed edges $(\sigma,\tau)$, such that $\tau\in\FV(\sigma)$.
The left-top panel in Figure~\ref{fig:mvf-example} shows an example of a multivector field consisting of seven multivectors (blue sets), 
    and the right-hand panel shows the corresponding graph $G_\cV$.
For a more detailed introduction to the multivector fields theory we refer to~\cite{LKMW2022}.

An \emph{isolating block} (or simply a \emph{block}) is a locally closed subset $B\subset X$ which is  a union of multivectors in $\cV$.
A partition $\cB=\{B_p\mid p\in P\}$ of $X$ is a \emph{block partition}~\cite{CMbarcodes2025} (or an acyclic partition~\cite{MroWan2025}) if each element $B_p\in\cB$ is an isolating block
    and there exists a partial order $(P, \leq)$ 
    such that presence of a path from $\sigma\in B_q$ to $\tau\in B_p$ in $G_\cV$ implies $p\leq q$.
We use this partial order called an \emph{admissible order} later 
to organize the boundary matrices.
A canonical block partition---\emph{the finest block partition}---is a collection of all strongly connected components of $G_\cV$.
The bottom panel of Figure~\ref{fig:mvf-example} presents the finest block partition with highlighted brown blocks.

Let $\cA$ and $\cB$ be families of subsets of $X$.
We say $\cA$ is a \emph{refinement} of $\cB$ if for every $A\in\cA$ there exists $B\in\cB$ so that $A\subset B$ and
denote this relation by $\cA\sqsubseteq\cB$.
If each element of $\cB$ is a union of at most two other elements of $\cA$ then we call $\cA$ an \emph{atomic refinement} of $\cB$.

Since a refinement of a multivector field can be interpreted as a perturbation of a combinatorial dynamical system (see~\cite[Section~4.4]{CMbarcodes2025}), 
a sequence of multivector fields by
\begin{align}\label{eq:sequence-of-multivector-fields}
    \zzV: \cV_0 \inovscr \cV_1 \inovscr \cV_2 \inovscr \ldots \inovscr \cV_T,
\end{align}
where $\cA \inovscr \cA'$ means that either $\cA\inscr\cA'$ or $\cA\ovscr\cA'$, 
    represents a parameterized combinatorial dynamical system.
It induces a sequence of the finest block partitions~\cite[Proposition~5.2]{CMbarcodes2025}:
\begin{equation}\label{eq:BD_zigzag_filtration}
    \zzBD: \BD_0 \inovscr \BD_1 \inovscr \BD_2 \inovscr \ldots \inovscr \BD_T,
\end{equation}
which we call a \emph{zigzag filtration of block partitions}.
In this paper we assume that at each step of $\zzBD$ we have an atomic refinement. 
However, the authors of~\cite{CMbarcodes2025} show how to reduce any filtration $\zzBD$ into a filtration with only atomic refinements.

Figure~\ref{fig:sequence-of-block-decompositions} (top row) shows five multivector fields on a 2-simplex forming a zigzag sequence 
    $\cV_0\ovscr\cV_1\ovscr\cV_2\inscr\cV_3\ovscr\cV_4$ of atomic refinements. 
The corresponding zigzag filtration of the block partitions is presented in the bottom row of Figure~\ref{fig:sequence-of-block-decompositions}. 
Note that we obtained a sequence of atomic refinements the same type: 
$\cB_0\ovscr\cB_1\ovscr\cB_2\inscr\cB_3\ovscr\cB_4$.
    


\section{From filtration of block partitions to modules and barcode}
\textbf{Indexing poset $\mathbb{P}$ for the filtration.}
Consider a zigzag filtration of block partitions~\eqref{eq:BD_zigzag_filtration}.
Denote the indexing set of $\cB_t$ by $P_t$, and 
    the elements of $\cB_t$ by $B_{p,t}$, where $p\in P_t$.
To denote refinements in both forward and backward directions, 
we use
the notation $t\pm 1$ which means either $t+1$ or $t-1$. By definition, 
    whenever $\cB_t\inscr\cB_{t\pm 1}$, for each $p\in P_t$ we have $q\in P_{t\pm 1}$ such that $B_{p,t}\subset B_{q,t\pm 1}$.
Let the corresponding inclusion be
    $\idxfwd{(p,t),(q,t\pm 1)}: B_{p,t} \hookrightarrow B_{q,t\pm 1}$. We write
    $\idxfwd{p,q}:=\idxfwd{(p,t),(q,t\pm 1)}$ when the dependence on $t$ is clear.
We have a poset $\mathbb{P}$ on elements $\bigcup_{t\in [0,T]}\{(p,t)\}_{p\in P_t}$ whose Hasse diagram is determined by
inclusions $\idxfwd{p,q}$. Essentially, in the Hasse diagram of $\mathbb{P}$, we have
$(p,t)\leq_{\mathbb{P}} (q,t\pm 1)$ if and only if 
    $\cB_t\inscr\cB_{t\pm 1}$ and $B_{p,t}\subset B_{q,t\pm 1}$.
The poset $\mathbb{P}$ is graded by $t\in [0,T]$ where we denote $\mathbb{P}_t=\{(p,t)\}_{p\in P_t}$.
The filtration~$\zzBD$ in~\eqref{eq:BD_zigzag_filtration} can alternatively
be viewed as a filtration of blocks indexed by the poset $\mathbb{P}$. 
Figure~\ref{fig:index-pairs-diagram} (left) shows this filtration for our running example.
    \cancel{
Globally we have a map
    $\idxfwd{t}:\bigoplus_{p\in\mathbb{P}_t}B_{p,t} \rightarrow \bigoplus_{q\in\mathbb{P}_{t+1}}B_{q,t+1}$ defined as 
    $\idxfwd{t}\coloneqq\bigoplus_{p\in{\mathbb{P}}_t}\idxfwd{p,t}$
    and $\idxbck{t}:\bigoplus_{q\in\mathbb{P}_{t+1}}B_{q,t+1} \rightarrow \bigoplus_{p\in\mathbb{P}_{t}}B_{p,t}$ defined as 
$\idxbck{t}\coloneqq\bigoplus_{p\in{\mathbb{P}}_t}\idxfwd{p,t}$,.

We write $\idxmap{t}$ when we do not specify the direction of the inclusion.

We rewrite~\eqref{eq:BD_zigzag_filtration} as
\begin{equation}\label{eq:BD_zigzag_filtration_2}
    \zzBD: 
        \bigoplus_{p\in\mathbb{P}_0} B_{p,0} \xlongleftrightarrow{\idxmap{0}} 
        \bigoplus_{p\in\mathbb{P}_1} B_{p,1} \xlongleftrightarrow{\idxmap{1}}  
        \ldots \xlongleftrightarrow{\idxmap{T-1}} 
        \bigoplus_{p\in\mathbb{P}_T} B_{p,T},
\end{equation}
which already looks like a standard zigzag filtration.
}

\afterpage{
\begin{figure}[ht]
    \centering
    \includegraphics[width=0.9\linewidth]{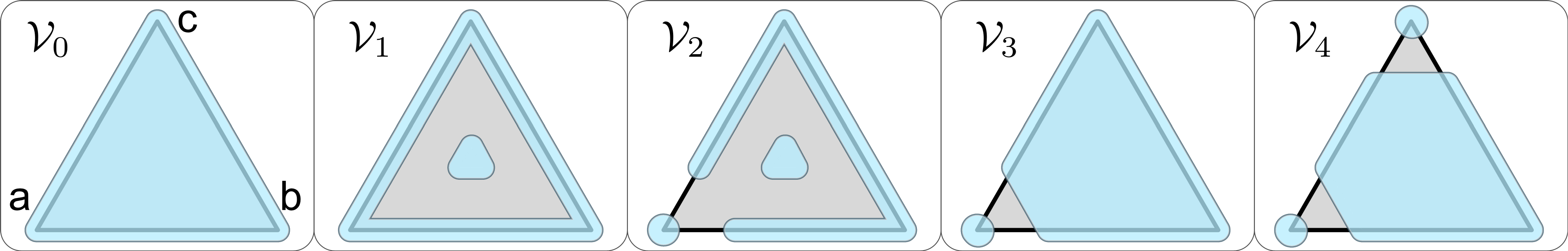}
    \vspace{0.05cm}
    
    \includegraphics[width=0.9\linewidth]{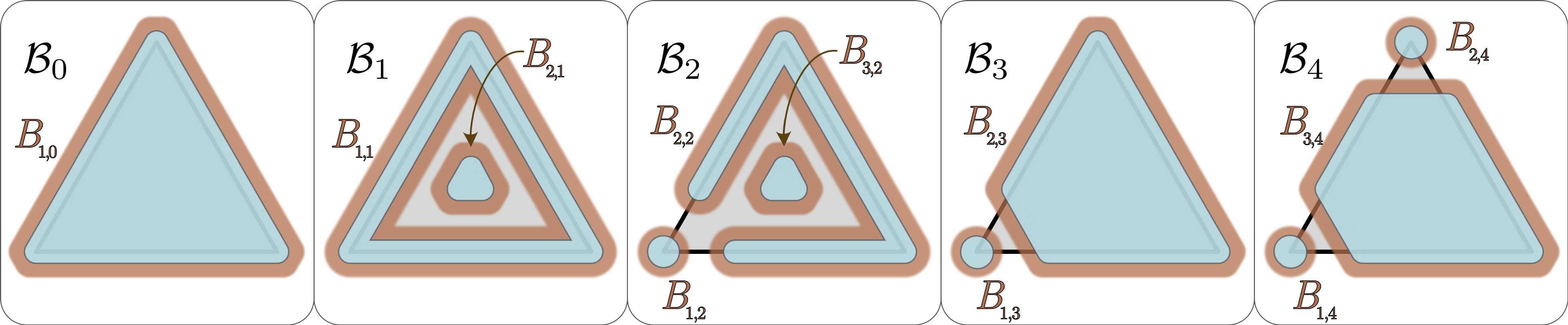}
    \caption{A parameterized multivector field (top) and the corresponding block partitions (bottom).}
    \label{fig:sequence-of-block-decompositions}
\end{figure}
\begin{figure}[ht]
    \centering
    \includegraphics[width=0.45\linewidth]{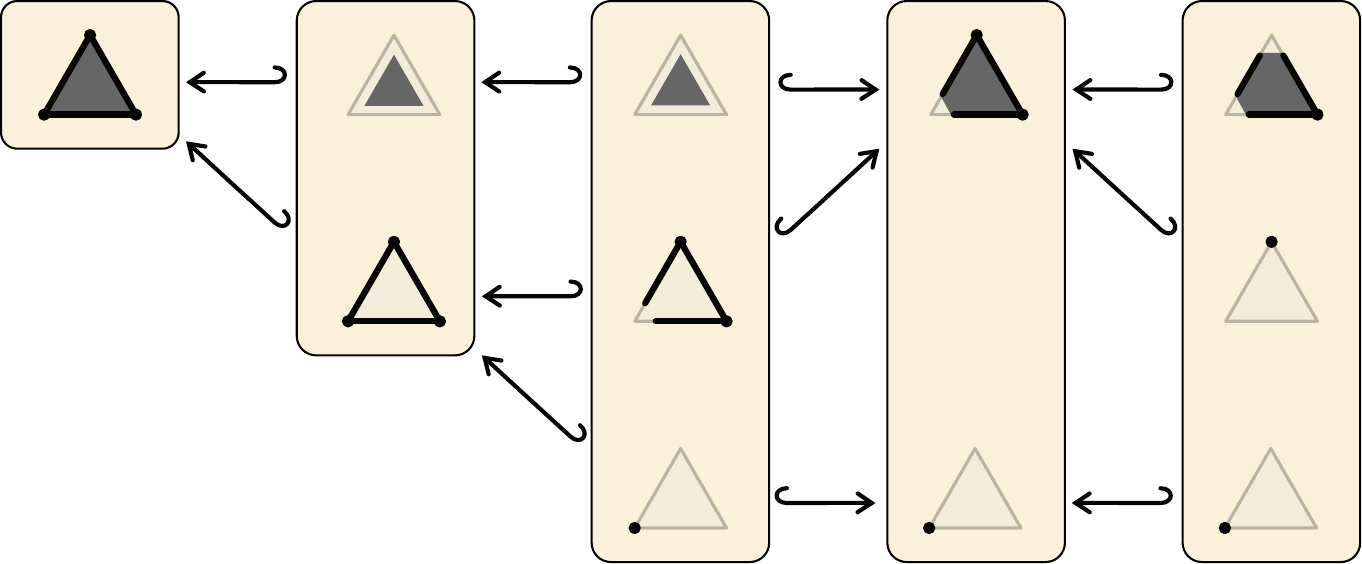}
    \hspace{0.5cm}
    \includegraphics[width=0.48\linewidth]{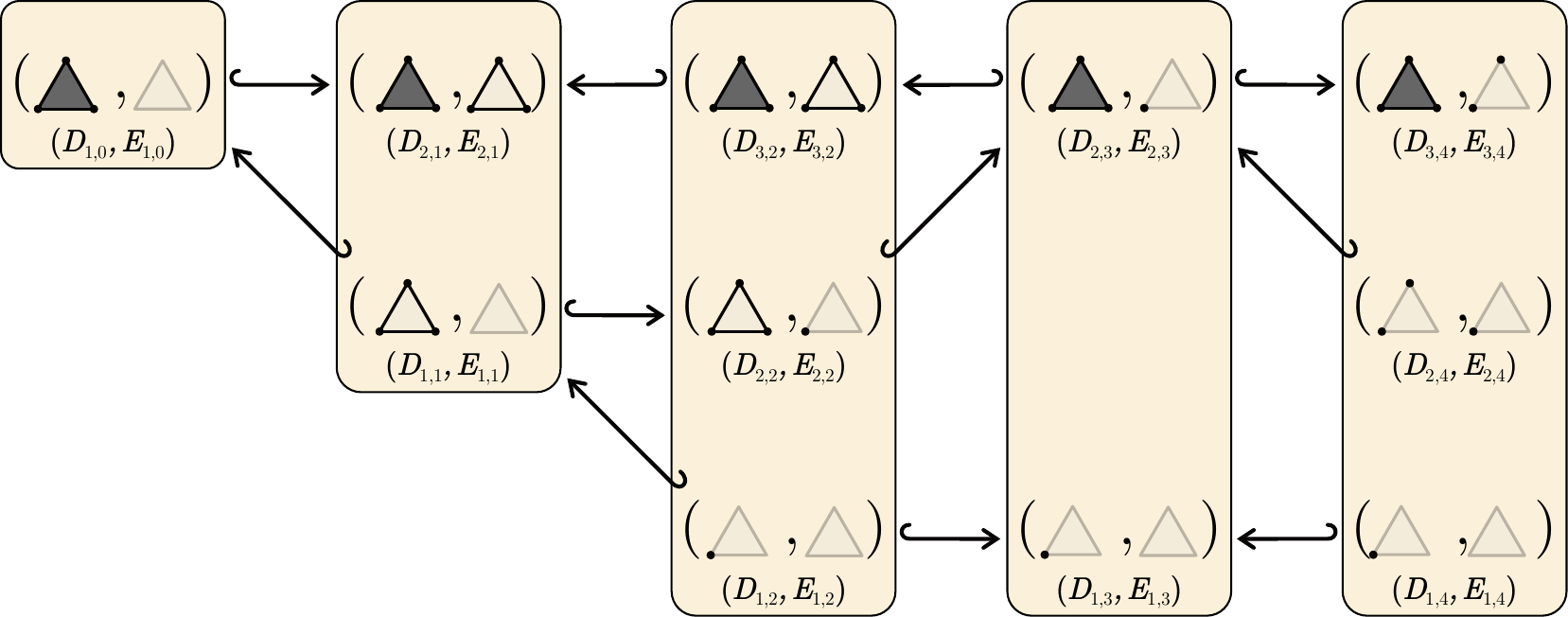}
    \caption{Left: Filtration of blocks from the example in Figure~\ref{fig:sequence-of-block-decompositions}.
        Right: Corresponding index pair's transition diagram for the example.
        }
    \label{fig:index-pairs-diagram}
\end{figure}
\begin{figure}[ht]
    \centering
    \includegraphics[width=0.47\linewidth]{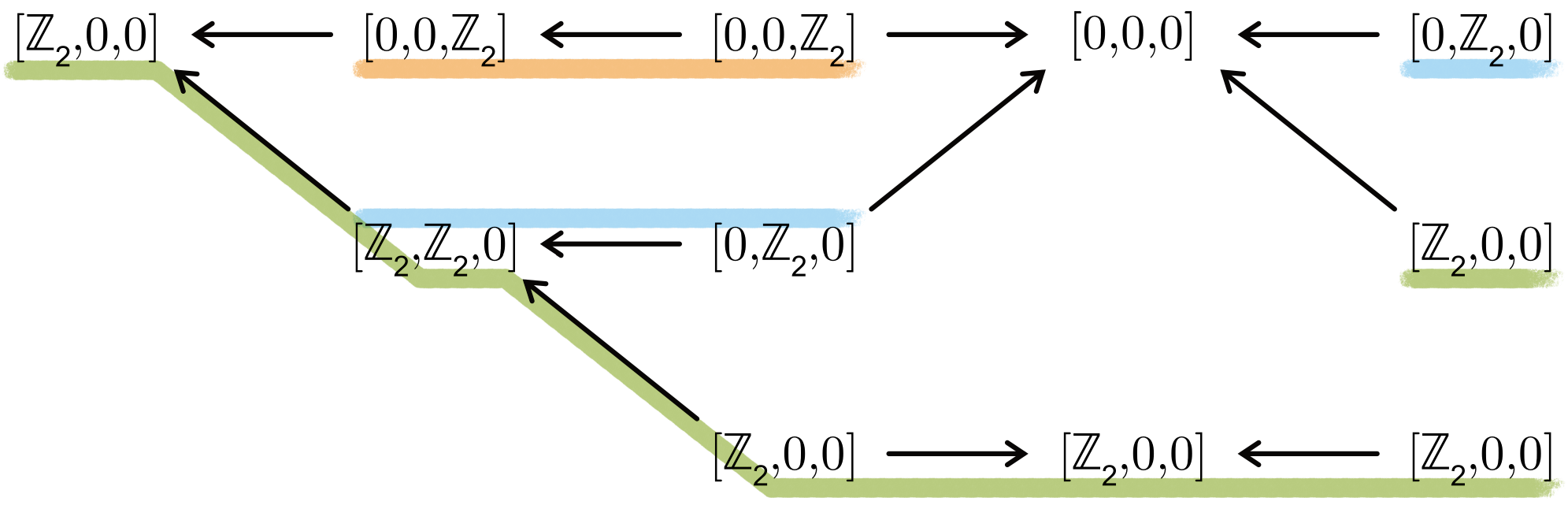}
    \hspace{0.5cm}
    \includegraphics[width=0.47\linewidth]{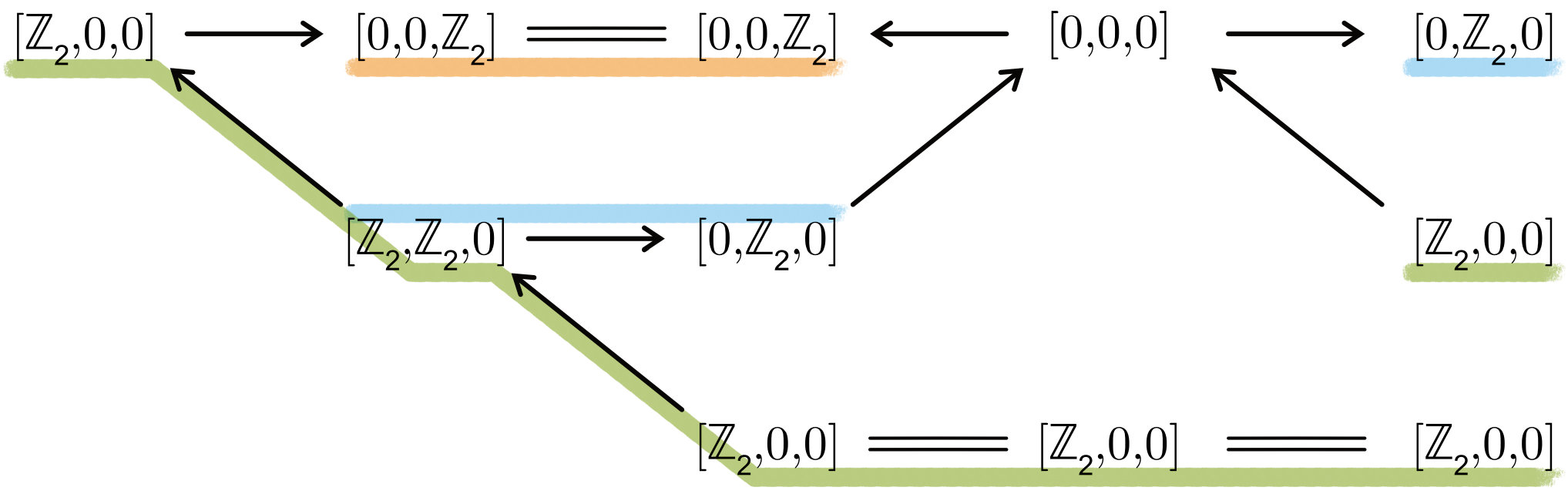}
    \caption{Conley-Morse persistence modules and barcodes induced by the filtration of blocks (left) and by the transition diagram (right) in Figure~\ref{fig:index-pairs-diagram}.
    The green bars correspond to $0$-homology, blue to $1$-homology, orange to $2$-homology.}
    \label{fig:conley-morse-persistence-module-and-barcode}
\end{figure}
}

By Proposition~\ref{prop:lefschetzhom} we can define the \emph{Conley index} of a block $B_{p,t}$ as $H(B_{p,t})$.
A typical TDA pipeline uses inclusions to obtain linear maps on homology, yielding a persistence module. 
This, however, does not work in our context
because the inclusion of a Lefschetz complex $B_{p,t}$ into $B_{q,t\pm 1}$ induces a well defined homomorphism only if  $B_{p,t}$ is closed in $B_{q,t\pm 1}$;
    otherwise, the cycles and boundaries in $B_{p,t}$ are not sent to the cycles and boundaries in $B_{q,t\pm 1}$.
Often $B_{p,t}$ is not closed in $B_{q,t\pm 1}$. 
For example, block $B_{2,1}$ in Figure~\ref{fig:sequence-of-block-decompositions}, consisting of the sole triangle, is not closed in $B_{1,0}$
from which it splits.
As a result, the triangle is a $2$-cycle in $B_{2,1}$, but 
not in $B_{1,0}$.
Therefore, a straightforward application of the homology functor on a filtration
of blocks indexed by $\mathbb{P}$ may not provide a persistence module with structural maps
induced by inclusions. 

\textbf{Indexing poset $\overline{\mathbb{P}}$ introduced in~\cite{CMbarcodes2025}.}
To overcome the above mentioned difficulty, the authors of~\cite{CMbarcodes2025}
    compute the Conley index using a pair of closed nested sets $(D_{p,t},E_{p,t})$ -- called an \emph{index pair} of a block $B_{p,t}$ -- such that $B_{p,t}=D_{p,t}\setminus E_{p,t}$. 
    Notice that, by Proposition~\ref{prop:lefschetz_relative_homology}, we have $H(B_{p,t})\cong H(D_{p,t},E_{p,t})$.
The inclusions among the index pairs across consecutive grades ($t$-values) may not match
the inclusions among the corresponding blocks, i.e., 
we may have inclusion $(D_{p,t},E_{p,t})\hookrightarrow (D_{q,t+1},E_{q,t+1})$
for index pairs whereas a reverse inclusion $B_{p,t} \hookleftarrow B_{q,t+1}$ for blocks.
For example, compare the inclusion arrows between the index pairs for $B_{2,3}$
and $B_{3,4}$ in Figure~\ref{fig:index-pairs-diagram} (right) and the arrows between $B_{2,3}$ and $B_{3,4}$ themselves in Figure~\ref{fig:index-pairs-diagram} (left).


Whenever a block $B$ splits into blocks $B_1$ and $B_2$, 
    one of them is necessarily closed in $B$ and the other is open in $B$.
Assume that $B_1$ is closed and $B_2$ is open.
As shown in \cite[Section~5.4]{CMbarcodes2025} one can always construct index pairs related by inclusions.
For instance, the triple $(N_2,N_1,N_0)\coloneqq(\cl B, \cl B_1\cup \mo B, \mo B)$ leads to a nested sequence of index pairs
    $(D_1, E_1)\coloneqq(N_1, N_0)$, 
    $(D, E)\coloneqq(N_2, N_0)$, 
    and 
    $(D_2, E_2)\coloneqq(N_2, N_1)$, for 
        $B_1$, $B$, and $B_2$, respectively.


Similarly to poset $\mathbb{P}$, now we have a poset $\overline{\mathbb{P}}$ graded by
$t\in [0,\bar{T}]$ that indexes the filtration of index pairs. The relative homology groups $H(D_{p,t},E_{p,t})$ of the index pairs, and the linear maps
induced by inclusions among them, provide a functor called the \emph{Conley-Morse persistence module}
$\overline{\mathcal G}=\overline{\mathcal G}(\zzBD):\overline{\mathbb{P}}\rightarrow \VecSp$ where $\VecSp$ is
the category of finite dimensional vector spaces over a fixed field $\field$. 
In~\cite{CMbarcodes2025}, it is shown  that 
$\overline{\mathcal G}(\zzBD)$ admits a decomposition (unique up to isomorphism)
$\overline{\mathcal G}(\zzBD)= \bigoplus_\pi \mathbb{I}_\pi$ where
each indecomposable $\mathbb{I}_\pi$ is a \emph{path module} defined as follows:
    a path $\pi$ in $\overline{\mathbb {P}}$ is a sequence $(p_0,t+0),(p_1,t+1),\ldots, (p_k,t+k)$ where
$p_i\in P_{t+i}$ for a $t\in [0,\bar T]$. A persistence module $ \mathbb{I}_\pi:\overline{\mathbb{P}}\rightarrow \VecSp$ is a path module with support
$\pi$ if 
$\mathbb{I}_\pi(x)\cong \Bbbk$ for
$x\in \pi$ and $0$ otherwise, and
the linear maps $\mathbb{I}_\pi(x\leq_{\overline{\mathbb{P}}} y)$ are
isomorphisms for $x,y\in \pi$ and zero maps otherwise.
The multiset\footnote{For elementary updates considered later, it is indeed a set, which we assume for simplicity in presentation.} of 
paths $\{\pi\}$ that supports the
path modules $\mathbb{I}_{\pi}$ 
is called the \emph{Conley-Morse persistence barcode} of filtration $\zzBD$.

The algorithm in~\cite{CMbarcodes2025} for computing the Conley-Morse persistence barcode proceeds
by considering every grade $t\in [0,\bar T]$ incrementally.
For each $t$, it 
iterates over multiple zigzag paths in $\overline{\mathbb{P}}$
and updates the barcode for each of the zigzag filtrations of index pairs
defined on these zigzag paths. 
The process becomes cumbersome because (i) it involves a `unwieldy' construction
of suitable index pairs, which generally
can be much larger than the corresponding blocks, (ii) multiple runs of the zigzag persistence 
algorithm are needed to cover the poset $\overline{\mathbb P}$ with zigzag paths, and
(iii) multiple index pairs are needed for each block to `glue'
index pairs at one grade to those at the next grade. 
Here we present an algorithm that overcomes all
these difficulties with a much simpler `vineyard'-like update algorithm.

\subsection{Conley-Morse persistence barcode from a simplified module}
\label{sec:construction-of-structural-map-h}
We already saw that the inclusions among blocks in a filtration over the
poset $\mathbb{P}$ does not provide a persistence module which prompted the
authors of~\cite{CMbarcodes2025} to define Conley-Morse persistence module
over a different poset $\overline{\mathbb{P}}$ graded by $t\in [0,\bar T]$ indexing `unwieldy' index pairs.
One of our key observations is that we can still get a persistence module over
a poset $\mathbb{P}$ graded the same way by $t\in [0,\bar T]$ indexing
blocks instead of index pairs through structural maps (homomorphisms)
$h_{p,q}: H(B_{p,t})\rightarrow H(B_{q,t\pm 1})$ which may not be induced directly by inclusions. 
This effectively \emph{reverses} some arrows in $\mathbb{P}$ from $\overline{\mathbb P}$.

To describe the induced homomorphisms $h_{p,q}$, consider
an inclusion
$B_{p,t}\hookrightarrow B_{q,t\pm 1}$. 
If $B_{p,t}=B_{q,t\pm 1}$ then $h_{p,q}:H(B_{p,t})\rightarrow H(B_{q,t\pm 1})$ is an isomorphism.
Otherwise, the two blocks are involved in a split (or equivalently merge in the opposite direction).
Let $\bl$ splits into blocks~$\bl_1$ and~$\bl_2$. 
One of the two sets is necessarily closed in $\bl$, the other is open; let us assume that~$\bl_1$ is closed and $\bl_2$ open in $\bl$.
We can construct a nested sequence of index pairs $(D_1,E_1)\xhookrightarrow{i'}(D,E)\xhookrightarrow{j'}(D_2,E_2)$ corresponding to $\bl_1$, $\bl$ and $\bl_2$, respectively.
These inclusions directly induce homomorphisms; therefore, we obtain the following diagrams:
\begin{equation}\label{eq:ar-split-inclusions}
    \begin{tikzcd}[row sep=tiny, column sep=normal]
        & \bl_2\\
        \bl  
            \arrow[ru, "\idxmap{2}", hookleftarrow, sloped]
            \arrow[rd, "\idxmap{1}", hookleftarrow, sloped, swap]& \\
        & \bl_1 
    \end{tikzcd}
    \hspace{1cm}
    \begin{tikzcd}[row sep=tiny, column sep=normal]
        & (D_2,E_2)\\
        (D,E)
            \arrow[ru, "j'", hookrightarrow, sloped]
            \arrow[rd, "i'", hookleftarrow, sloped, swap]& \\
        & (D_1,E_1)
    \end{tikzcd}
    \hspace{1cm}
    \begin{tikzcd}[row sep=tiny, column sep=small]
        & & H(D_2,E_2) \ar[r,equal] & V'\\
        X' \ar[r,equal] & H(D,E) \arrow[ru, "j_\ast'",sloped] & \\
        & & H(D_1,E_1)\arrow[lu, "i_\ast'",swap, sloped] \ar[r,equal]& W'
    \end{tikzcd}
\end{equation}
where $X'\coloneqq H(D,E)$, $W'\coloneqq H(D_1,E_1)$, and $V'\coloneqq H(D_2,E_2)$. 
Since images and kernels constitute direct summands of respective vector spaces,
    these homology vector spaces split as $W'=W_1'\oplus W_2'$ with 
    $W_2'\coloneqq \ker i_\ast'$, $X'=X_1'\oplus X_2'$ with $X_1'\coloneqq\im i_\ast'$,
    and $V'=V_1'\oplus V_2'$ with $V_2'\coloneqq\im j_\ast'$ 
where the summands $W_1'$, $X_2'$, and $V_1'$ are not uniquely determined. We work with a fixed choice for them.

{\bf Splits for blocks}: 
Similarly, define vector spaces
$X\coloneqq H(\bl)$, $W\coloneqq H(\bl_1)$, and $V\coloneqq H(\bl_2)$. 
By Proposition~\ref{prop:lefschetz_relative_homology}, we have
isomorphisms $X\stackrel{\psi}{\cong}X'$, $W\stackrel{\phi_1}{\cong}W'$, and 
$V\stackrel{\phi_2}{\cong}V'$.
By defining 
    $i_\ast\coloneqq\psi\circ i_\ast'\circ \phi_1^{-1}$ and 
    $j_\ast\coloneqq\phi_2\circ j_\ast'\circ \psi^{-1}$ 
    (see diagram \eqref{eq:lefschetz_ar_split_joint}), 
    we get similar splits $W=W_1\oplus W_2$ with 
    $W_2\coloneqq\ker i_\ast$, $X=X_1\oplus X_2$ with $X_1\coloneqq\im i_\ast$,
    and $V=V_1\oplus V_2$ with $V_2\coloneqq\im j_\ast$. 
Here, also the summands $W_1$, $X_2$, and $V_1$ are not uniquely determined. 
We fix them, by choosing a basis. Independent of this choice,
proof of~\cite[Theorem~5.12]{CMbarcodes2025} provides:
\begin{proposition}\label{prop:structure-of-i-j-ast}
    Restrictions $i_\ast|_{W_1}:W_1\rightarrow X_1$ and $j_\ast|_{X_2}:X_2\rightarrow V_2$ are isomorphisms.  
    Moreover, $i_\ast|_{W_2}=0$ and $j_\ast|_{X_1}=0$.
\end{proposition}

\begin{equation}\label{eq:lefschetz_ar_split_joint}
    \begin{tikzcd}[row sep=1.5em, column sep=1.5em]
        & & & H(B_2)\ar[r,equal]& V_1\oplus V_2\\
        & & & H(D_2, E_2)\arrow[u, "\varphi_2"]\ar[r,equal]& V_1'\oplus V_2'\\
        X_1\oplus X_2\ar[r, equal] & H(B) \arrow[rruu, bend left, "j_\ast"]\arrow[r,leftarrow, "\psi"]
            & H(D,E)\arrow[ru, "j_\ast'"] & \ar[l, equal]X_1'\oplus X_2' \\
        & & & H(D_1, E_1)\arrow[d, "\varphi_{1}"]\arrow[lu, "i_\ast'"]\ar[r,equal]& W_1'\oplus W_2'
            \\ 
        & & & H(B_1)\arrow[lluu, bend left, "i_\ast"] \ar[r,equal]& W_1\oplus W_2
    \end{tikzcd}
\end{equation}

{\bf Choosing bases and spaces}: 
For the spaces $W_2$, $X_1$, and $V_2$, we fix bases $S_{W_2}$, $S_{X_1}$, and $S_{V_2}$, respectively; extend them to bases of the total spaces 
$W$, $X$, and $V$, which define the spaces $W_1$, $X_2$, and $V_1$ with
complement subbases
$S_{W_1}=S_W\setminus S_{W_2}$, $S_{X_2}=S_{X}\setminus S_{X_1}$, and $S_{V_1}=S_{V}\setminus S_{V_2}$,
respectively.
Let $S_{W_2}$ be a basis for $\ker i_\ast=W_2$ and extend it to $S_{W}$
to have $S_{W_1}$.
Using $i_\ast|_{W_1}$ as an isomorphism (Proposition~\ref{prop:structure-of-i-j-ast}), define
    the basis of $X_1$ as $S_{X_1}=\{i_\ast(w)\,|\, w\in S_{W_1}\}$.
    Extend it to $S_{X}$ to have $S_{X_2}$.
Using $j_\ast|_{X_2}$ as an isomorphism (Proposition~\ref{prop:structure-of-i-j-ast}), define the basis $S_{V_2}=\{j_\ast(x)\,|\, x\in S_{X_2}\}$
and extend it to $S_{V}$ to have $S_{V_1}$.

{\bf Defining structural maps}: 
Now, we are ready to define structural maps corresponding to $\idxmap{1}$ and $\idxmap{2}$ in diagram~\eqref{eq:ar-split-inclusions}.
For the inclusion $\idxmap{1}:\bl_1\hookrightarrow \bl$, 
we have an induced homomorphism $i_*:H(\bl_1)\rightarrow H(\bl)$ and we write $g\coloneqq i_*$ to indicate that $\bl_1$ is closed in $\bl$.

The inclusion $\idxmap{2}:\bl_2\hookrightarrow \bl$ does not induce a homomorphism when $\bl_2$ is not closed in $\bl$.
In this case, we indirectly define $f: V\rightarrow X$ as follows: $\forall v\in S_{V_1}$, 
    $f(v)=0$ and
    $\forall v\in S_{V_2}$, $f(v)=x$ where $x\in S_{X_2}$ and $j_\ast(x)=v$. 
Essentially, $f$ reverses the map $j_\ast$ via the isomorphism $j_\ast|_{X_2}$.

Thus, we have three types of maps $h_{p,q}:H(B_{p,t})\rightarrow H(B_{q,t\pm 1})$:
    when there is no split and $B_{p,t}=B_{q,t\pm 1}$, then $h_{p,q}$ is a straightforward isomorphism;
    when $B_{p,t}$ is the `closed' part of a split of $B_{q,t\pm 1}$ ($\bl_1\hookrightarrow\bl$ type of inclusion),
        then $h_{p,q}=g$, which is directly given by the inclusion;
    and when $B_{p,t}$ is the open part of a split of $B_{q,t\pm 1}$ 
        ($\bl_2\hookrightarrow\bl$ type),
        then $h_{p,q}=f$.
The following diagram summarizes the discussion:
\begin{equation}\label{eq:lefschetz_ar_split_regular}
    \begin{tikzcd}[row sep=tiny, column sep=2.5cm]
        & H(\bl_2)=V_1\oplus V_2\arrow[ld, "f=0\oplus (j_\ast|_{X_2})^{-1}",sloped]\\
        X_1\oplus X_2 = H(\bl) & \\
        & H(\bl_1)=W_1\oplus W_2\arrow[lu, "g=(i_\ast|_{W_1})\oplus 0",swap, sloped]
    \end{tikzcd}
\end{equation}

\paragraph{Connecting the splits.} 
It may happen that two consecutive refinements split a block differently.
In this case, as in~\cite{CMbarcodes2025}, we connect the two splits of the
same block by an isomorphism $\gamma$, which essentially changes the basis as needed. Algorithmically, $\gamma$ is implicitly implemented with an elementary operation called
\textbf{Transposition}. For simplicity, we omit showing these isomorphisms in the Figures~\ref{fig:index-pairs-diagram} and~\ref{fig:conley-morse-persistence-module-and-barcode}.
See Section~\ref{sec:appendix_connecting_the_splits} for an example of a non-trivial $\gamma$.

\begin{tikzcd}[row sep=tiny, column sep=2.0cm]
    \hat{V_1}\oplus \hat{V_2}
        \arrow[rd, "\hat{f}=0\oplus (j_\ast|_{\hat{X}_2})^{-1}", sloped]
        & & &
        V_1\oplus V_2
            \arrow[ld, "f=0\oplus (j_\ast|_{X_2})^{-1}", sloped] \\
    &
    \hat{X}_1\oplus \hat{X}_2 = H(\bl) 
        \arrow[r, "\gamma"]
    &
    H(\bl) = X_1\oplus X_2
    \\
    \hat{W}_1\oplus \hat{W}_2
        \arrow[ru, "\hat{g}=(i_\ast|_{\hat{W}_1})\oplus 0", swap, sloped]
        & & &
        W_1\oplus W_2
            \arrow[lu, "g=(i_\ast|_{W_1})\oplus 0", swap, sloped]
\end{tikzcd}

\paragraph{Decomposition.}
We define persistence module $\mathcal{G}: \mathbb{P}\rightarrow \VecSp$ indexed
by the poset $\mathbb{P}$ where $\mathcal G((p,t))=H(B_{p,t})$ and the structural maps are given by $\mathcal G((p,t)\leq_{\mathbb{P}} (q,t\pm 1))=h_{p,q}$ 
The reader can easily verify that the structure of arrows in $\PP$ guarantees that there are no commutativity issues. 
    In particular, there are no two alternative paths between any two nodes (see Figure~\ref{fig:index-pairs-diagram}, left).
Even though the
two modules $\overline{\mathcal G}$ and
$\mathcal G$ are defined on two posets $\overline{\mathbb{P}}$ and $\mathbb{P}$ respectively, with possibly different arrow directions, they decompose into path modules with the same support. Proposition~\ref{prop:CMmodule-simplified}
below states this fact which follows from Proposition~\ref{prop:composition}.

\begin{proposition}
    Let $M:\mathbb X\rightarrow \VecSp$ be a persistence module on a finite
    poset $\mathbb X$ and let $\{\mathbb{I}_\alpha\}_{\alpha\in \Lambda}$ be a set of submodules of $M$ with
    $\bigoplus_{\alpha\in\Lambda}\mathbb{I}_\alpha(x)=M(x)$  $\forall x\in \mathbb{X}$.
    Then, $M=\bigoplus_{\alpha\in\Lambda} {\mathbb I}_\alpha$.
    \label{prop:composition}
\end{proposition}
\begin{proof}
We need to show that $\mathbb{I}_\alpha$ is a direct summand of $M$, or equivalently, 
    the inclusion $i_\alpha\colon\mathbb{I}_\alpha\hookrightarrow M$ 
    admits a module homomorphism (a splitting) $j_\alpha: M\rightarrow \mathbb I_\alpha$ with $j_\alpha\circ i_\alpha=\mathrm{id}_{\mathbb I_\alpha}$. 
Since for each $x\in \mathbb{X}$, there exists $\Lambda'\coloneqq\{x\mid \mathbb{I}_\alpha(x)\neq 0\}\subseteq \Lambda$ 
    such that $M(x)=\bigoplus_{\alpha\in \Lambda'} \mathbb{I}_\alpha(x)$,
there are unique linear projections $j_{\alpha,x}\colon M(x)\rightarrow \mathbb{I_\alpha}(x)$,
which pick out the 
$\alpha$-summand in the direct-sum decomposition at $x$.
Define $j_\alpha$ to be the collection of these maps $(j_{\alpha,x})_{x}$ over all
$x\in \mathbb{X}$.

We need to check these $j_{\alpha,x}$ assemble to a morphism of persistence modules, i.e.,
they commute with the structure maps denoted $\phi_{x\leq y}\colon M(x)\rightarrow M(y)$.
Take any $x\leq y$ and any $v\in M(x)$. 
Write the decomposition at $x$ as
$v=\sum_{\beta\in \Lambda'} v_\beta$ where $v_\beta\in \mathbb{I}_\beta(x)$.
Since $\mathbb{I}_\beta$ is a submodule, we get $\phi_{x\leq y} (v_\beta)\in \mathbb I_\beta(y)$. Therefore, $\phi_{x\leq y}(v)=\sum_\beta \phi_{x\leq y}(v_\beta)$
gives the decomposition of $\phi_{x\leq y}(v)$ into the summands at point $y$.
Now apply the projections:
\[
j_{\alpha,y}(\phi_{x\leq y}(v))=j_{\alpha,y}\big(\phi_{x\leq y}\big(\sum_{\beta\in \Lambda'} v_\beta\big)\big)
= \phi_{x\leq y}(v_\alpha)=\phi_{x\leq y}(j_{\alpha,x}(v))
\]
Thus for every $x\leq y$, we have
\[
j_{\alpha,y}\circ \phi_{x\leq y}= \phi_{x\leq y}\circ j_{\alpha,x}.
\]
So, the family $(j_{\alpha,x})_{x}$ is a natural transformation, that is, a persistence module homomorphism $j_\alpha\colon M\rightarrow \mathbb I_\alpha$ is well defined.

By construction, for each $x$, the restriction of $j_{\alpha,x}$ to $\mathbb I_\alpha(x)$
is the identity on $\mathbb I_\alpha(x)$ (because projections in a direct-sum decomposition do that). Therefore, $j_\alpha\circ i_\alpha=\mathrm{id}_{\mathbb I_\alpha}$ as we are required to prove. \qed
\end{proof}

\begin{proposition}
    Let $\overline\Pi$ be the set of paths in $\overline{\mathbb{P}}$ so that 
    $\overline{\mathcal G}=\bigoplus_{\pi\in \overline{\Pi}} \bar{\mathbb{I}}_{\pi}$.
    Then, ${\mathcal G}=\bigoplus_{\pi\in \Pi} \mathbb{I}_{\pi}$
    for a set of paths $\Pi$ in $\mathbb{P}$ where $\Pi=\overline{\Pi}$.
    \label{prop:CMmodule-simplified}
\end{proposition}
\begin{proof}
    First, we recognize that both $\mathbb{P}$ and $\mathbb{\overline{P}}$ have the same 
    set of objects (points) and a one-to-one correspondence between their relations in their
    Hasse diagrams (transition diagrams). The only difference is that certain relations in these
    Hasse diagrams have opposite directions, namely, an inclusion $(D_{p,t},E_{p,t})\hookrightarrow (D_{q,t\pm 1},E_{q,\pm 1})$ in index pairs may correspond to an inclusion $B_{p,t}\hookleftarrow B_{q,t\pm 1}$ in blocks.

    Consider any summand $\bar{\mathbb{I}}_\pi$ in the decomposition $\overline{\mathcal G}=\bigoplus_{\pi\in \overline{\Pi}} \bar{\mathbb{I}}_{\pi}$. Let $(p,t)\rightarrow (q,t\pm 1)$ be
    any edge in the Hasse diagram of $\overline{\mathbb{P}}$. 
    We have $\bar v_p\stackrel{\bar\kappa_\ast}{\rightarrow}\bar v_q$
    where $\bar v_p=\bar{\mathbb{I}}_\pi(p,t)$ and $\bar v_q=\bar{\mathbb{I}}_\pi(q,t\pm 1)$ are zero or one dimensional vector spaces with the structural map $\bar\kappa_\ast$ in $\overline{\mathcal G}$
        induced by inclusions in index pairs. 
    Now, we will show that that the summand $\bar{\mathbb{I}}_\pi$ is consistent with the module $\mathcal G$ defined on~$\mathbb{P}$. 
    Recall the definition of the structural maps $h_{p,q}$ for $\mathcal G$. We have two cases.

    \textbf{No reversal.} $(p,t)\rightarrow (q,t\pm 1)$ is
    a relation in $\mathbb P$: by definition 
        $h_{p,q}= i_\ast=\psi\circ i_\ast'\circ\phi_1^{-1}$
    where $\psi$ and $\phi_1$ are isomorphisms (see Diagram~\ref{eq:lefschetz_ar_split_joint}).
    Since $\bar{\kappa}_\ast=i_\ast'$ we have $h_{p,q}(\phi_1(\bar v_p))=\psi(i_\ast'(\bar v_p))=\psi(\bar v_q)$.
    Writing
    $\phi_1(\bar v_p)=v_p\in W$ and $\psi(\bar v_q)=v_q\in  X$, we get $h_{p,q}(v_p)=v_q$.

    \textbf{A reversal.} $(p,t)\leftarrow (q,t\pm 1)$ is a relation in $\mathbb P$: 
        If $\bar v_p\in X_1'$, we have $j'_\ast(\bar v_p)=0$ (by Proposition~\ref{prop:structure-of-i-j-ast}), 
            and also $v_p\not\in\im h_{q,p}=\im f$ by construction of $f$. 
    If $\bar v_p\in X_2'$, then by definition,
        $h_{q,p}(\phi_2(\bar v_q))=\psi(\bar v_p)$ because $j_\ast'(\bar v_p)=\bar v_q$ which agrees with $\bar{\kappa}_\ast$.
        Then, we have $h_{q,p}(v_q)=v_p$.
        

    It follows that for each summand $\bar{\mathbb{I}}_\pi$ in $\bar{\mathcal G}$, we have a submodule $\mathbb{I}_\pi$ of $\mathcal G$ with the same support. Furthermore, because of the isomorphisms $\psi,\phi_1,\phi_2$,
    $\bigoplus_{\pi\in \Pi} (v_p=\mathbb I_\pi(p,t))=\bigoplus_{\pi\in \Pi}(\bar v_p= \overline{\mathbb{I}}_\pi(p,t))=\overline{\mathcal G}(p,t)=\mathcal{G}(p,t)$.
    Therefore, by Proposition~\ref{prop:composition}, $\mathcal G=\bigoplus_{\pi\in \Pi} \mathbb I_\pi$. By the uniqueness of decomposition up to isomorphism,
    any other decomposition $\mathcal G=\bigoplus_{\pi\in \Pi'}\mathbb I_\pi$ must have
    $\Pi=\Pi'$. \qed
\end{proof}
~\\
Proposition~\ref{prop:CMmodule-simplified} allows us to work with the simplified
module $\mathcal G$ in place of $\overline{\mathcal G}$ for computing the Conley-Morse
persistence barcode. However, we need to fix the  summands $W_1$, $X_2$, and~$V_1$ which our algorithm
does by choosing bases for them. Using them,
we compute the set of path modules satisfying the condition
of Proposition~\ref{prop:composition}, which form a decomposition of $\mathcal G$.
For a path $\pi=(p_b,b),(p_{b+1},b+1),\ldots,(p_d,d)$ in~$\mathbb{P}$, consider
a sequence of classes $L=[u^b],[u^{b+1}],\ldots,[u^{d-1}],[u^d]$ where $u^b,\ldots,u^d$
are cycles in Lefschetz complexes $B_{p_b,b}\in {\mathcal B}_b,\ldots,B_{p_d,d}\in {\mathcal B}_d$ respectively.
For $t\in [b,d-1]$, we write $[u^t]\leftrightarrow [u^{t+1}]$ if $h_{p,q}([u^t])=[u^{t+1}]$ for some $p\in P_t$ and $q\in P_{t+1}$, or $h_{p,q}([u^{t+1}])=[u^t]$ for $p\in P_{t+1}$ and $q\in P_t$ (depending on the direction of the filtration). 

\begin{definition} We say $L$ is
a \emph{representative} for $\pi$ if for every $t\in [b,d-1]$, we have
(i) $[u^t]\leftrightarrow [u^{t+1}]$, and (ii) for $b\not = 0$, $H({\mathcal B}_{b-1})\ni 0\leftrightarrow [u^b]$, and for $d\not = \bar T$, $[u^d]\leftrightarrow 0\in H({\mathcal B}_{d+1})$. Also, we denote the corresponding path module
as $\mathbb{I}_\pi$, where $\mathbb I_\pi(p_t,t)=[u^{t}]$.
\end{definition}
The following proposition is a reformulation of Proposition~\ref{prop:composition} in terms of representatives which the algorithm works with.
\begin{proposition}
    Let $\{L_k=[u_k^{b_k}],[u_k^{(b+1)_k}],\ldots,[u_k^{(d-1)_k}],[u_k^{d_k}]\}_{k\in \Lambda}$ 
    be a set
    of representatives for
    respective paths $\Pi=\{\pi_k=(p_{b_k},b_k),\ldots,(p_{d_k},d_k)\}_{k\in \Lambda}$    
    where $\forall (p,t)\in \mathbb{P}$, there exists some index set $\Lambda'\subseteq \Lambda$ so that $H(B_{p,t})=\bigoplus_{k\in \Lambda'}[u_k^{t}]$. 
    Then, $\mathcal G=\bigoplus_{\pi_k\in \Pi} \mathbb{I}_{\pi_k}$.
    \label{prop:representative}
\end{proposition}
\cancel{
A direct characterization of map $f$ in terms of chains is not straightforward, however here we provide key properties.
By $\partial_B$ we mean the boundary operator restricted to $B$, and by $[c]_V$ the homology class of cycle $c$ in $V$.

\begin{proposition}\label{prop:reversing-the-split-map}
    Let $c\in\Zgroup(B_2)$ such that $c\neq 0$.
    The following statements hold:
    \begin{enumerate}[label=\arabic*)]
        \item for any $e\in\Cgroup(B_1)$ such that no subchain of $e$ is in $\Zgroup(B)$ and $d\coloneqq c+e\in\Zgroup(B)$ and $[d]_X\in X_2$, 
            we have $f([c]_{V})= [d]_X$,
        \item 
            $f([c]_{V})\not = 0$ $\Leftrightarrow$
            $[\partial_B c]_W=0$ $\Leftrightarrow$ 
            there exists $e\in\Cgroup(B_1)$ such that $d\coloneqq c+e\in\Zgroup(B)$.
    \end{enumerate}
\end{proposition}
}


\begin{figure}
    \centering
    \includegraphics[width=0.8\linewidth]{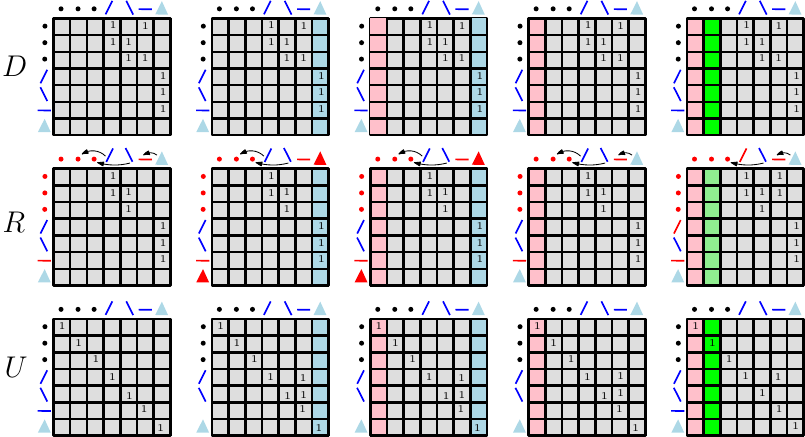}
    \caption{Matrix updates for filtration in Figure~\ref{fig:sequence-of-block-decompositions}.
    (upper row) Boundary matrix $D$ with blocks; (middle row) reduced matrix $R$, simplices for non-homogeneous columns are colored red; homogeneous-target pairs are indicated by arrows; (lower row) uniform matrix $U$.
    }
    \label{fig:MatDecompose}
\end{figure}

\section{Conley-Morse persistence barcode from matrix decompositions}\label{sec:matrix-decomposition}
Although our algorithm works for any fixed field $\Bbbk$, 
    for simplicity we describe it for $\Bbbk=\mathbb{Z}_2$.
Let $\zzBD\coloneqq\{\cB_t\}_{t\in[0,T]}$ be a zigzag filtration of block partitions.
To compute Conley index of a block $B_{p,t}$ we can simply perform the reduction algorithm on the boundary matrix restricted to that subset and read the homology.
However, in order to keep track of the changing Conley indices we keep them all in one matrix. 
Namely, for each $t$ we assume to have a square $n\times n$ \emph{block boundary matrix} $D_t$ of $X$ with $n$ cells ordered according to a linearization of admissible order $\leq$ for $\BD_t$, 
    that is, a cell $\tau\in B_{p,t}$ comes before a cell $\sigma\in B_{q,t}$ whenever $p<q$, or $p=q$ and $\tau$ is a face of $\sigma$ (see Figures~\ref{fig:MatDecompose}, \ref{fig:LS}, \ref{fig:RS}).
    Let $c_1,\ldots,c_n$ denote the left-to-right ordered set of columns and
    $r_1,\ldots, r_n$ denote the top-to-bottom ordered set of rows in $D_t$.
    Each column $c_i$ and row $r_i$ corresponds to a cell $\sigma_i\in X$ for which we sometimes write $c_i=c_{\sigma_i}$ and $r_i=r_{\sigma_i}$. 
    A block $B\in \BD_t$ is represented in $D_t$ by the columns 
    $\{c_i\}_{\sigma_i\in B}$ and \emph{all} rows $\{r_i\}_{\sigma_i\in X}$ though 
    we say $c_i$ and $r_i$ \emph{belong} to the unique block that contains $\sigma_i$. 
    The column $c_i$ initially represents the chain $\partial \sigma_i$.
Then, block-independent reductions of~$D_t$ such as the ones in Algorithm~\ref{alg:block-reduction} may change the corresponding chain.
These reductions use 
left-to-right column additions mimicking the well known persistence algorithm.
To understand the process, 
assume that $R_t$ is a matrix with columns and rows indexed and partitioned
with blocks, as in $D_t$.

The \emph{pivot} $\low(c_i)$ of a column $c_i$ in matrix $R_t$ is the row corresponding to the lowest non-zero entry in $c_i$ if $c_i$ is non-empty and is undefined otherwise.
We say that the column $c_i$ and the row $r_i$, 
    which correspond to the same cell $\sigma_i$, 
    are \emph{homogeneous} if $\low(c_i)$ is defined and both $c_i$ and $\low(c_i)$ belong to the same block $B$. 
Two homogeneous columns $c_i$ and $c_j$ in a block $B$ are in \emph{conflict} if $\low(c_i)=\low(c_j)$.
We call $c_i$ \emph{reduced}
if either $c_i$ is non-homogeneous or $c_i$ does not conflict with any other
column in $B$. Notice that, pivots of reduced homogeneous
columns over all blocks are globally unique; 
but, the same may not be true for non-homogeneous columns. 
The matrix
$R_t$ is said to be \emph{block reduced} if all its columns are reduced.

\begin{algorithm}
    \caption{{\sc{BlockIndependentReduction}}}\label{alg:block-reduction}
    \KwData{An $n\times n$ matrix $D$---block boundary matrix for $\BD$} 
    \KwResult{$R=DU$ block reduction of $D$}
    $R\coloneqq D$\;
    $U\coloneqq I_n$\tcp*{an $n\times n$ identity matrix}\
    \For(){$j:=1$ to $n$}{ 
             \If{$c_j=R[\cdot,j]$ is homogeneous}{
                \While {$\exists$ \text{$s<j$ \& $\low(c_s)=\low(c_j)$}}
                   {$R[\cdot,j]=R[\cdot,s] +R[\cdot,j]$;\\
                   $U[\cdot,j]=U[\cdot,s] +U[\cdot,j]$;
                }
            }
    }
\Return{$R$ and $U$}
\end{algorithm}
Algorithm \ref{alg:block-reduction} reduces each block in $\cB$ independently, similarly to the standard reduction algorithm. 
Within each block only homogeneous columns are added.
It is easy to verify that matrix $R$ obtained as an output is block reduced, and matrix $U$ is uniform.
As discussed in Proposition~\ref{prop:pointwiseBasis}, by identifying all non-homogeneous and non-targeted columns in each block $B\in\cB$ in $R$ we obtain the Conley indices.

From an input block boundary matrix $D_t$, we can obtain a block reduced matrix
$R_t$ by left-to-right column additions.
If a homogeneous column $c_i$ conflicts with a column $c_j$ to its left,
we update $c_i$ by the left-to-right addition $c_i\leftarrow c_i+c_j$. We continue
these additions until $c_i$ becomes reduced. At any generic step of this process,
if $R_t$ denotes the matrix obtained so far, we can write $R_t=D_tU_t$ where $U_t$ is an upper triangular square matrix. 
Each column $c_i$ and each row $r_i$ in $D_t$
    correspond to a column in $U_t$ that we denote with $u_i$ and a row in $U_t$ which we still denote with $r_i$.
If we do not allow any other type of column addition to make
$R_t$ block reduced, $U_t$ becomes \emph{uniform}, as defined below. 

 \begin{definition}
		Given a block boundary matrix $D_t$ for $\md_{t}$, let $R_t$ be a matrix where $R_t=D_tU_t$ for an upper triangular matrix $U_t$. We say a column $u_\ell$ in a block $B\in \md_t$ of the matrix $U_t$ is \emph{uniform} if $\forall i\not=\ell$,
        $U_t[i,\ell]\not=0$ $\implies$ $r_i$ belongs to block $B$ and $r_i$ is homogeneous.
        This means that all reductions in $R_t$ occur with the left-to-right additions between homogeneous columns in the same block only. We call $U_t$ \emph{uniform} if all of its columns are uniform.
 \end{definition}
\begin{proposition}
        A matrix $R_t=D_tU_t$ is obtained with left-to-right column additions between
        homogeneous columns in the same block if and only if $U_t$ is upper triangular and uniform.
        \label{prop:umatrix}
\end{proposition}

The Conley index of a block $B$ in a block reduced matrix $R_t$ is implicitly
given by the classes of chains represented by certain columns of $U_t$. 
In this
respect, we introduce targeted columns, which play an important role in the algorithm, too. A column $c_j$ in $R_t$ is called \emph{targeted} if $r_j=\low(c)$ for
some homogeneous column $c$ in $R_t$. In what follows, we take the liberty to denote
the chain given by a column $c$ and $u$ in $R_t$ and $U_t$, respectively, as $c$ and $u$ themselves. 
With this understanding, we have $\partial u_i=c_i$ for every $i\in [1,n]$.

\begin{proposition}
Let $R_t=D_tU_t$ be block reduced and $U_t$ be uniform. Let $\Lambda$ be the index set
where $\{c_\ell\}_{\ell\in \Lambda}$ be the set of all
non-homogeneous and non-targeted columns in a block $B\in \md_t$ in $R_t$.
Then, the set $\{[u_\ell]\}_{\ell\in \Lambda}$ constitutes a basis of the Conley index $H(B)$.
\label{prop:pointwiseBasis}
\end{proposition}
\begin{proof}
     First, recall that $H(B)= H(\cl B,\mo B)$ by Proposition~\ref{prop:lefschetzhom}. The block $B$ in block reduced matrix $R_t$ has
     every homogeneous column with a unique pivot row. The columns of $B$ in $U_t$ represent a basis of the chain space of $B$.
     By construction,
     for every column $c_\ell$ in $R_t$ and its corresponding column
     $u_\ell$ in $U_t$ we have $\partial u_\ell=c_\ell$. Since $U_t$ is
     uniform, the chain $u_\ell$ is in the chain space of $B$.
     The chain $u_\ell$ is a relative cycle if and only if $\partial u_\ell=c_\ell$ is
     in $\mo B$ which happens if and only if
     the column $c_\ell$ in reduced $R_t$ is non-homogeneous. Let $\Lambda'=\{\ell\,|\,c_\ell \mbox{ is non-homogeneous in } B\}$.
     
     Further, the relative cycles $\{u_\ell\}_{\ell\in \Lambda'}$ 
     are independent
     because each $u_\ell$ contains a unique simplex at the diagonal of $U_t$.
     Thus, the relative cycles $\{u_\ell\}_{\ell\in \Lambda'}$ constitute a basis of the relative cycle space $\Zgroup(\cl B,\mo B)$. 
     Among these cycles, the targeted columns correspond to boundaries. 
     To see this, let $c_\ell$ be a targeted column of a homogeneous column $c_{\ell'}$ in $B$. The
     chain $\partial u_{\ell'}=c_{\ell'}=x+y$ is a cycle where $x$ and $y$ are the maximal subchains
     present and not present in $B$ respectively. 
     Then, $y$ is in $\mo B$ and since mouth of a locally closed set is closed, $\cl y\subset \mo B$. 
     We have $\partial x=\partial y \subset \cl y \subset \mo B$ which means $\partial x\subset \mo B$.
     We can replace the chain $u_\ell$ with $x$ and replace $c_\ell$ with $\partial x$. 
     Then, $R_t$ remains block reduced and $c_\ell$ remains non-homogeneous and
     targeted.
     However, the relative cycle $x=u_\ell$ is a relative boundary of
     $u_{\ell'}$. It follows that the classes of relative cycles 
     $\{[u_\ell]\,| \,c_\ell \mbox{ is non-homogeneous and non-targeted}\}$
     constitute a basis for $H(\cl B,\mo B)$. In other words, if $\Lambda\subseteq \Lambda'$
         is the set of indices $\{\ell |\, c_\ell \mbox{ is non-homogeneous and non-targeted}\}$,
     then $\{[u_\ell]\}_{\ell\in \Lambda}$ constitutes a basis of $H(B)$.\qed       
     \end{proof}
~\\
Our algorithm for computing Conley-Morse persistence barcode builds on the following approach:
	 We obtain a block
    reduced matrix $R_0$ and uniform matrix $U_0$ from $D_0$ by block-independent reductions say by Algorithm
\ref{alg:block-reduction}.
	Then, instead of computing $R_{t+1}$ and $U_{t+1}$ directly with the algorithm from $D_{t+1}$, we obtain them by updating $R_t$ and $U_t$, respectively.
This allows us to track representatives through the filtration and eventually provides the Conley-Morse persistence barcode according to Proposition~\ref{prop:representative}. 
Any update from $D_t$ to $D_{t+1}$ can be expressed as a sequence of the following four elementary operations.

\begin{enumerate}[label=(\arabic*)]
    \item \textbf{Left/Right Split}; the order of columns remains unchanged, but a block is split on the 
    left/right boundary by a column. 
    If $c_1,c_2,\ldots, c_m$ is the order of the columns for a block, then 
       a left split produces two blocks, 
        one with the single column $c_1$
        and the other with the columns $c_2,\ldots, c_m$. 
    Similarly, 
    a right split produces two blocks, one consisting of columns $c_1,c_2,\ldots, c_{m-1}$ and the other with the single column $c_m$. 
    
    \item \textbf{Left/Right Merge}; this operation is the opposite of the split operation. 
        The order of columns remains unchanged, but a block is merged with a block consisting
        of a single column $c$ immediately to its left for left merge and
        immediately to its right for a right merge.
    \item \textbf{Transposition}; two columns are exchanged within a single block.
\end{enumerate}

A filtration $\zzBD$ of block partitions is called \emph{elementary} if the
block boundary matrix $D_{t+1}$ for the block partition $\md_{t+1}$ can be obtained
from the block boundary matrix $D_t$ for the block partition $\md_t$ by an
elementary update.
In particular, our running example (Figure~\ref{fig:sequence-of-block-decompositions}) is elementary. Theorem~\ref{thm:elementary-expansion} in Section~\ref{sec:elementary-expansion} justifies the generality of elementary operations.
    
\subsection{Tracking representatives with matrix decompositions}
Our algorithm computes the Conley-Morse persistence barcode by tracking a representative for each
bar. This is done by associating each bar with a non-homogeneous and non-targeted column $c_\ell^t$ of the matrix $R_t$ and determining if the class $[u_\ell^t]$ of the column $u_\ell^t$ in $U_t$ still maps to a non-zero class at $t+1$.
   
    Let $R_t=D_tU_t$ and $c_\ell^t=c_{\sigma_\ell}$ be a column of a $(p-1)$-chain in $R_t$ corresponding to the simplex $\sigma_\ell$ and $u_\ell^t$ be the corresponding column of a $p$-chain in $U_t$. We have $\partial u_\ell^t=c_\ell^t$.

Combining Proposition~\ref{prop:representative} and Proposition~\ref{prop:pointwiseBasis}, we get
the following result.

\begin{theorem}
Let $\{L_k=[u_k^{b_k}],[u_k^{(b+1)_k}],\ldots,[u_k^{(d-1)_k}],[u_k^{d_k}]\}_{k\in\Lambda}$ 
    be a set
    of representatives for
    respective paths $\Pi=\{\pi_k=(p_{b_k},b_k),\ldots,(p_{d_k},d_k)\}_{k\in \Lambda}$    
    where $\forall t\in [0,T]$, the classes $\{[u_k^{t}]\}_{k\in \Lambda'}$
    for an index set $\Lambda'\subseteq \Lambda$ be such that the set $\{c_k^t\}_{k\in \Lambda'}$ is the set of all
    non-homogeneous and non-targeted columns in $R_t$. Then $\mathcal G=\bigoplus_{\pi_k\in \Pi} \mathbb{I}_{\pi_k}$.
    \label{thm:barcode}
\end{theorem}

Theorem~\ref{thm:barcode} is the key to our algorithm. It says that if we track the non-homogeneous and non-targeted columns of block reduced matrices $R_t$s', we can track the Conley-Morse persistence barcode of a given filtration $\zzBD$. Algorithmically, under each operation we need to 
keep the matrix $R_t$ block reduced, $U_t$ uniform, and keep track of
non-homogeneous, non-targeted columns. 
The tracking means that we account
for (i) birth, that is, if a new non-homogeneous, non-targeted column appears,
(ii) continuation, that is, to which a non-homogeneous and
non-targeted column in $R_t$ continues in $R_{t+1}$,
and (iii) death, that is, if a non-homogeneous and non-targeted column becomes
homogeneous or targeted in $R_{t+1}$.


We show how we do this tracking under different elementary operations without recomputing the matrices $R_t$ and $U_t$ from scratch and updating them efficiently. 
The matrix $U_t$ helps keep the matrix $R_t$ block reduced with left-to-right additions (Proposition~\ref{prop:umatrix}). The time complexity of each operation varies ($O(n^2)$ in the worst case)
which we mention while describing them.
The tracking of representatives is
drawn from the result below which justifies why the classes implicitly selected
by the algorithm for a representative follow its definition. 
Recall the definitions
    of spaces 
    $X$, $V$, and $W$ and their splits
for a split of a block $\bl=\bl_1\sqcup \bl_2$ where $\bl_1$
is necessarily closed in $\bl$. 
Also, denote $f=h_{p,q}$ when its domain is $V$ and $g=h_{p,q}$ when its domain is $W$. We drop indices $(p,q)$ when there is no confusion.
\begin{proposition}\label{prop:reversing-the-split-map}
    The following statements hold:
    \begin{enumerate}[label=\arabic*)]
        \item\label{it:reversing-the-split-map-B2}
            Let $u\in \Zgroup(B_2)$ such that $[u]\in S_{V}$.
            If $\exists e\in\Cgroup(B_1)$ such that $d\coloneqq u+e\in\Zgroup(B)$,
            $[d]\in S_{X_2}$, 
            then $f([u])=[d]$, otherwise $f([u])=0$.
        \item\label{it:reversing-the-split-map-B1} 
            Let $u\in \Zgroup(B_1)$ and $u\not=0$. If $[u]\in S_{W_1}$,
            then $g([u])$ is a basis element in $S_{X_1}$ and if $[u]\in S_{W_2}$ then
            $g([u])=0$.
    \end{enumerate}
\end{proposition}
\begin{proof}
    To see \ref{it:reversing-the-split-map-B2} 
    let $(D,E)$, $(D_2,E_2)$ be index pairs so that $B=D\setminus E$ and $B_2=D_2\setminus E_2$. Since $[d]\in S_{X_2}$, we have $d\in\Cgroup(B)=\Cgroup(D\setminus E)$. Then, $[d]\in S_{X_2'}$ as well.
    By inclusion of $(D,E)\hookrightarrow (D_2,E_2)$, we have $d\in \Cgroup(D_2)$.
    Since $e\in \Cgroup(E_2)$ we have $[d]_{V_2'}=[u]_{V_2'}$ where
    $j_\ast'([d])=[u]$ and therefore $f([u])=[d]$ by definition of $f$. If there
    is no $d$ with the stated conditions, we do not have a preimage
    of $[u]$ by $j_\ast$. Then, by definition, we have $f([u])=0$.
    


    \ref{it:reversing-the-split-map-B1} 
    Follows directly from the definitions.\qed

\end{proof}

\section{Split and Merge}\label{sec:split-and-merge}
In a split/merge, a single column is moved out of/into a block.
Let $B:=B_1\sqcup B_2$ be the block being split or merged upon where $B_1$ is necessarily closed in $B$.
Then, according to discussion in Section~\ref{sec:construction-of-structural-map-h}, 
the map $f: H(B_2)\rightarrow H(B)$ is not induced
by inclusion.

In a split/merge, the orders of the columns and rows are maintained, $U_t$ remains upper triangular but may not remain uniform because homogeneity of the columns in $R_t$ may get disturbed. 
Then, appropriate column additions are made in $R_t$ to get it reduced and corresponding additions are made in $U_t$ to restore its uniformity. 
In a split, exactly one of (i) or (ii) holds for a column $c$ which is in the blocks involved with the split and is associated to a bar in the Conley-Morse persistence barcode:
    (i) a new bar associated to $c$ is born, 
    (ii) an existing bar associated to $c$ continues possibly with a change in representative. 
For a merge, similarly exactly one of (i) or (ii) holds:
    (i) the existing bar associated to $c$ dies, 
    (ii) the existing bar associated to $c$ continues possibly with a change of representative.

The birth and death happen in pairs, that is, a pair of bars gets born if birth happens in a split and a pair of bars die if death happens in a merge. All
bars associated to columns in blocks not involved in split/merge continue with the same
representative. 

\begin{figure}[t]
    \centering
    \includegraphics[width=0.23\linewidth]{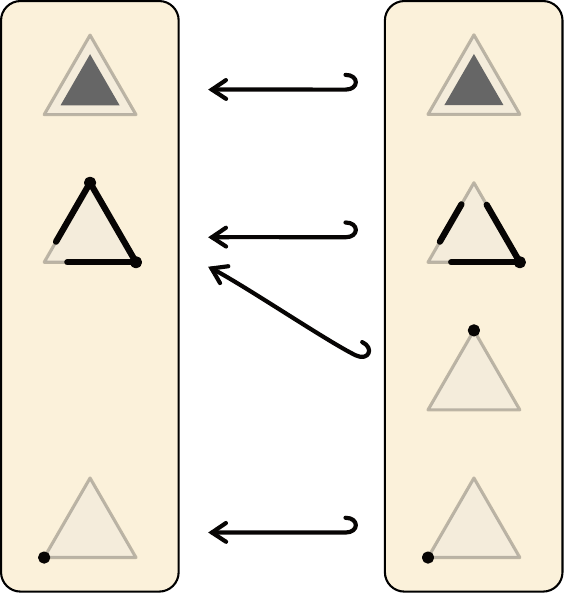}
    \hspace{0.6cm}
    \includegraphics[width=0.6\linewidth]{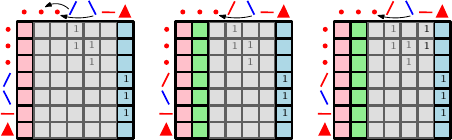}
    \caption{(left) Reduced matrix $R$ with three blocks; (middle) leftmost vertex of the gray block is split to create a new block (green) (LS.birth case), it also causes an edge (newly marked red) in the gray block to become non-homogeneous; (right) the column of the new red edge
    is added to the rightmost red edge column to `undo' the prior addition. A left merge
    will be opposite.
    }
    \label{fig:LS}
\end{figure}

\subsection{Left Split (LS)}
Let $c_j$ be the column split on left from the ordered set of columns $c_j,c_{j+1},\ldots,c_m$ for a block. Observe that $c_j$ must be non-homogeneous in $R_t$ and remains non-homogeneous
in $R_{t+1}$ as a block consisting of a single
simplex, say $\sigma$. Also, 
no other column is added to $c_j$ in both $R_t$ and $R_{t+1}$ keeping the corresponding column $u_j$ uniform in both $U_t$ and $U_{t+1}$.

\textbf{Case(LS.continue):} If $c_j$ is non-targeted in $R_t$, it remains so in $R_{t+1}$, 
    and thus the bar associated with $c_j$, along with all other live bars, continues
    from index $t$ to index $t+1$. 
Also, there are no changes in the representatives of the bars. This case takes
$O(1)$ time.

Writing $u_j^t:=u_j$ in $U_t$ and $u_j^{t+1}:=u_j$
in $U_{t+1}$, we see that chains $u_j^t$ and $u^{t+1}_j$ are equal to the simplex $\sigma$ where
$c_j=c_\sigma$. We have $\{\sigma\}$ closed in $B$ because
it has no face in $B$. 
Taking $B_1=\{\sigma\}$ we have $h([\sigma])=g([\sigma])=[\sigma]$ (see \eqref{eq:lefschetz_ar_split_regular}). 
For $B_2=B\setminus \{\sigma\}$, apply Proposition~\ref{prop:reversing-the-split-map}.\ref{it:reversing-the-split-map-B2} with
$e=0$ to see that $h([u_k])=f([u_k])=[u_k]$ for any non-homogeneous
and non-targeted column $u_k$ in $B_2$.

\textbf{Case(LS.birth):} If $c_j$ was targeted by a homogeneous column $c_k$ in $R_t$, the column $c_j$ becomes non-targeted and $c_k$ becomes non-homogeneous in $R_{t+1}$. Then, to restore
uniformity of $U_{t+1}$, we add $c_k$ ($u_k$ resp.) to each column $c_\ell$ ($u_\ell$ resp.) to which it had been added in $R_t$ ($U_t$ resp.) ($u_\ell$ in $U_t$ contains a non-zero entry in the row $r_k$). 
This step takes time $O(n^2)$. Notice that each column $c_\ell$
is non-homogeneous in $R_t$ because its pivot row index is lower
than the pivot row index of $c_k$ which is $j$ (the lowest row index belonging to $B$). 
After adding $c_k$, the pivot row of each such non-homogeneous column $c_\ell$ becomes the same as that of $c_k$. 
Since $c_k$ is non-homogeneous in $R_{t+1}$, column $c_\ell$ becomes non-homogeneous in $R_{t+1}$. 
To restore the property that $R_{t+1}=D_{t+1}U_{t+1}$, we add the column $u_k$ to each column $u_\ell$, which also restores the uniformity of $U_{t+1}$. 
This case causes a change in the Conley-Morse persistence barcode. The non-homogeneous column $c_j$ cannot be targeted in $R_{t+1}$. Also, no homogeneous column
in $R_t$ could have targeted the homogeneous column $c_k$ (see Proposition~\ref{prop:pivot-hom}) which becomes non-homogeneous in $R_{t+1}$ remaining non-targeted.
Therefore, we introduce two new bars at the index $t+1$ associated to the columns $c_j$ (in degree, say $p-1$) and $c_k$ (in degree $p$) with the representatives
$[u_j]$ and $[u_k]$, respectively. 

The chain $u_j=\sigma$ is in $B_1$ which was a boundary in $B$ because
$c_j$ was targeted before the split,
giving
$h([u_j])=g([u_j])=0$. The chain $u_k$ is in $B_2$ which was homogeneous in $B$ giving $h([u_k])=f([u_k])=0$. 
For every non-homogeneous and non-targeted column $c_\ell$ that changed
representative from $[u^t_\ell]$ to $[u^{t+1}_\ell]=[u^t_\ell+u_k]$,
observe that $[u^t_\ell]\in S_{V_2}$ and by Proposition \ref{prop:reversing-the-split-map}.1), $f([u^t_\ell])=[u^t_\ell]$ where $d=u^t_\ell$ and $e=0$. 
Then, $f([u^{t+1}_\ell])=f([u^t_\ell]+[u_k])=f([u^t_\ell])=[u^t_\ell]$.

\subsection{Right Split(RS)}
Let $c_m$ be the column split on right from the ordered columns $c_1,c_2,\ldots,c_m$ for block $B$. The column $c_m$
necessarily becomes non-homogeneous in $R_{t+1}$.
However, the column $u_m$ in $U_{t}$ may no longer be uniform in $U_{t+1}$.
To enforce it, we reset the column $u_m$ with a single non-zero entry at the diagonal and reset $c_m$ with $\partial \sigma$ where $c_m$ was $c_\sigma$ for a simplex $\sigma$. With this change, $U_{t+1}\leftarrow U_{t}$ becomes uniform
while keeping $R_{t+1}=D_{t+1}U_{t+1}$ block reduced. 

\textbf{Case(RS.continue):} $c_m$ is non-homogeneous in $R_t$. It is necessarily non-targeted.  The Conley-Morse persistence barcode does not change qualitatively,
    except that every live bar is extended to the index $t+1$
    and the representative of the bar associated to $c_m$ changes
from $[u_m^t]:=[u_m]$ in $R_t$ to $[u_m^{t+1}]:=[u_m]=[\sigma]$ in $R_{t+1}$.

The set $B\setminus \{\sigma\}$ is necessarily closed in $B$ because $c_m=c_\sigma$ being the rightmost column, $\sigma$ cannot be a face of any simplex in $B\setminus\{\sigma\}$.
We set $B_1=B\setminus\{\sigma\}$ and $B_2=\{\sigma\}$.
We can apply Proposition~\ref{prop:reversing-the-split-map}.\ref{it:reversing-the-split-map-B2} taking
$d=u^t_m$ and $e=u^t_m\setminus \sigma$ to claim that
$h([u^{t+1}_m])=f([\sigma])=[u^t_m]$. All other non-homogeneous non-targeted
columns $c_\ell$ in $R_t$ remain unchanged in $R_{t+1}$.
For such a column, we have the chain $u_\ell$ in $B_1$
giving $h([u_\ell])=g([u_\ell])=[u_\ell]$.
\begin{figure}
    \centering
    \includegraphics[width=0.27\linewidth]{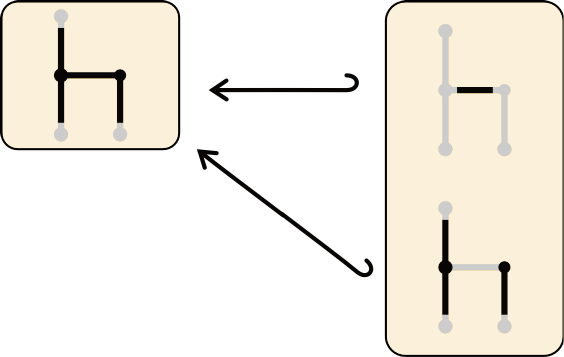}
    \hspace{0.5cm}
    \includegraphics[width=0.6\linewidth]{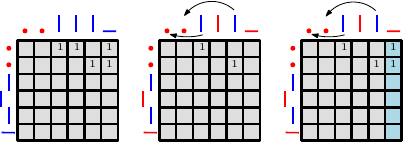}
    \caption{(left) Matrix $D$ with one block; (middle) reduced matrix $R$, homogeneous-target pairs are indicated with arrows; (right) a right split (RS.continue case), the non-homogeneous column after the split is reset to the boundary to undo prior addition; A right merge will be opposite.
    }
    \label{fig:RS}
\end{figure}

\cancel{
\tamal{this para would be useful later to show that $h$ is indeed
a reversal of the map originally considered in~\cite{CMbarcodes2025}.}
Notice that, as Lefstchez chains, $u_m^{i+1}=\sigma$ and $u_m^i=u'+\sigma$ both
in a block $B_{p,i}$. The inclusion between the relative pairs
$\iota: (u_m^i,0)\hookrightarrow (u_m^i,u')$ induces the homomorphism 
$\iota_*:H(u_m^i,0)\rightarrow H(u_m^i,u')$. We can assume that $\partial \sigma\not=0$ in $B_{p,i}$ because then $u_m^i=\sigma$ (no column is added to $c_m$ in $R_t$ to reduce it) and $u_m^i=u_m^{i+1}$ and there is nothing to prove. With the assumption
of $\partial \sigma\not=0$, we have $\partial u'\not =0$ in $B_{p,i}$. Therefore,
$H(u')=0$ allowing us to apply Lemma 24.4\cite{Munkres} by taking $K=L=u_m^i$, $K_0=0$, and $L_0=u'$ and to conclude $\iota_*:H(u_m^i,0)\rightarrow H(u_m^i,u')$ is an
isomorphism. By excision, we also have $H(\sigma)=H(u_m^i,u')$. Thus, we have
representative classes for the bar $b$ before and after the update satisfy
$[u_m^i]=[u_m^{i+1}]$ in the common space $B_{p,i}$.
}

\textbf{Case(RS.birth):} $c_m$ is homogeneous in $R_t$. This case triggers a qualitative change in the Conley-Morse persistence barcode, namely, we create two new bars starting at the index $t+1$. One of these bars in degree $p$ corresponds to the new
non-homogeneous column $c_m=c_\sigma$ in $R_{t+1}$ if $\sigma$ is a $p$-simplex.
The other bar in degree $p-1$ corresponds to the column $c_k$ that was targeted by the homogeneous column $c_m$ in $R_t$ 
because $c_k$ becomes non-targeted in $R_{t+1}$. We have chains
$u^{t+1}_m=\sigma$ in $B_2$ and $u^t_m$ in $B$. 
We have $[(\partial\sigma)|_{B_1}]\not=0$
because 
$c_m$ is homogeneous in $R_t$. Then, there is no $e\in \Cgroup(B_1)$ 
where
$\sigma+e\in \Zgroup(B)$. Then, by Proposition~\ref{prop:reversing-the-split-map}.\ref{it:reversing-the-split-map-B2}, we have $h([u^{t+1}_m])=f([\sigma])=0$. For the same reason
as in Case(RS.continue), for all other non-homogeneous non-targeted columns $c_\ell$,
we have $h([u_\ell])=g([u_\ell])=[u_\ell]$.

Both Case(RS.continue) and Case(RS.birth) take $O(n)$ time because they mainly involve changing a single column in $R_t$ and $U_t$.

\subsection{Left Merge (LM)}
Let $c_j=c_\sigma$ be the sole column in a block $B_1$ (has a single cell $\sigma$) that is merged on the left of a block $B_2$ with the ordered set of columns $c_{j+1},\ldots, c_m$ to become a new merged block $B$ consisting of the ordered set of columns $c_j,c_{j+1},\ldots, c_m$.
The column $c_j=c_\sigma$ was necessarily non-homogeneous in $R_t$ 
with its corresponding column
$u_j$ being uniform, which contains a single non-zero entry in the diagonal of $U_t$. The column $c_j$ remains non-homogeneous in $R_{t+1}$ and the column $u_j$ remains uniform. All homogeneous columns in
$R_t$ remain homogeneous and reduced. 
The left merge mirrors the left split and thus has two cases. 

\textbf{Case(LM.continue):} Column $c_j=c_\sigma$ is not targeted  
after merge. The bar associated with $c_j$ with the representative $[u_j^t]=[\sigma]$ remains associated with $c_j$ after merging with the 
representative $u^{t+1}_j=[\sigma]$. Just as in Case(LS.continue), 
one can show $h([\sigma])=g([\sigma])=[\sigma]$ and for all other non-homogeneous and
non-targeted columns $c_k$ in $B_2=B\setminus \{\sigma\}$, we have $h([u_k])=f([u_k])=[u_k]$. This case takes $O(1)$ time.

\textbf{Case(LM.death):} Column $c_j$ is
    targeted after the merge and let $c_k$ in $B$ targets $c_j$.
Then, $c_k$ becomes homogeneous after the merge along with possibly
other columns $c_{i_1},\ldots, c_{i_\ell}$ in $B$ where $\low(c_{i_1})=\cdots=\low(c_{i_\ell})=\low(c_k)$.
Observe that every addition $c_{i_j}\leftarrow c_k+c_{i_j}$, $j\in \{1,2,\ldots,l\}$, makes $c_{i_j}$ non-homogeneous. This is because the index of $\low(c_{i_j})$ after the addition becomes smaller than the index of $\low(c_k)$, which is $j$.
Thus, updated columns that were non-homogeneous in $R_t$ remain non-homogeneous
after the addition with $c_k$. We also add the corresponding columns 
in $U_t$. At this point, $R_{t+1}\leftarrow R_t$ becomes
block reduced and $U_{t+1}\leftarrow U_t$ becomes uniform. 
The bar $b$ associated with the column $c_j$ and the bar $b'$ associated
with the column $c_k$ in $R_t$ stop at index $t$ because $c_k$ becomes homogeneous targeting $c_j$ in $R_{t+1}$. Due to $O(n)$ column additions, this case
takes $O(n^2)$ time.

The chain $u_j^t=\sigma$ in $B_1$ becomes a boundary in $B$  
as in Case(LS.birth) and thus $h([u_j^t])=g([u_j^t])=0$. Similarly,
the argument in Case(LS.birth) can be applied to claim
$h([u_k])=f([u_k])=0$ and 
$h([u^{t+1}_\ell])=[u^t_\ell]$ for
every non-homogeneous, non-targeted column $c_\ell$ that changed
representative from $[u^t_\ell]$ to $[u^{t+1}_\ell]=[u^t_\ell+u_k]$.

\subsection{Right Merge (RM)}
Let $c_m=c_\sigma$ be the sole column in a block merged to the right of a block with the ordered set of columns $c_j,\ldots, c_{m-1}$ to become a new merged block $B$ consisting of the ordered set of columns $c_j,c_{j+1},\ldots, c_m$. After merging, we reduce the column $c_m$ within $B$ and update $U_{t}$ accordingly. 
This makes $R_{t+1}\leftarrow R_t$ reduced and $U_{t+1}\leftarrow U_t$ uniform, which do not require other matrix updates.
Right merge mirrors the right split giving two cases.

\textbf{Case(RM.continue)}: $c_m=c_\sigma$ remains non-homogeneous after the merge and reduction. 
Then, no other updates are required.
Similarly to the right split, we can argue that the bar associated with $c_m$ 
remains associated with it. Specifically, the representative
$[u^t_m]$ changes to the representative $[u^{t+1}_m]$ where $h([u^{t+1}_m])=[u^t_m]$ and
for all other non-homogeneous non-targeted columns $c_\ell$, $u^t_\ell$ remain
unchanged, giving $h([u^t_\ell])=g([u^t_\ell])=[u^{t+1}_\ell]$.

\textbf{Case(RM.death):} $c_m=c_\sigma$ becomes homogeneous after the merge and reduction. 
The column
$c_m$ was necessarily non-homogeneous and non-targeted in $R_t$ as a single column 
in its block.
This means $c_m$ was associated with a bar, say $b$, in the barcode. When $c_m$ becomes homogeneous after the merge, the bar $b$ and the bar $b'$ associated with its targeted column $c_k$ ends at the index $t$. This case takes $O(n^2)$ time for $O(n)$
column additions used to reduce $c_m$. Mirroring Case(RS.birth), we 
can claim $h([u^t_m])=f([\sigma])=0$ and $h([u_k])=g([u_k])=0$ and representatives
for all other bars remain unchanged.
\section{Transposition}
\label{sec:transpose-shuffle}
In transposition, two columns in a block $B$ are exchanged and thus the spaces $B$ and
$H(B)$ do not change.
Let $c_j$ and $c_{j+1}$ be two columns in a block $B$ that are transposed. Corresponding
rows $r_j$ and $r_{j+1}$ are also transposed. Depending on cases as detailed below, we take different actions. We will appeal to the following result which follows from the proof of Proposition 4.3 in~\cite{DeHaLi2025}
in the context of reductions for connection matrices. However, it
applies in the context of block reductions as proposed here.

\begin{proposition}
    Let $c$ be any column in a block reduced matrix. If $r_k$ is the pivot row
    of $c$, then $c_k$ is a non-homogeneous column.
    \label{prop:pivot-hom}
\end{proposition}

In all cases below, we first perform the column/row exchanges in both matrices
$R_t$ and $U_t$. The columns $c_j$ and $c_{j+1}$ are exchanged in $R_t$ and
the corresponding columns $u_j$ and $u_{j+1}$ are exchanged in $U_t$. The corresponding
rows $r_j$ and $r_{j+1}$ are exchanged in both $R_t$ and $U_t$. Observe that 
with these exchanges, we still have the property that $R_t=D_{t+1}U_t$ ($D_{t+1}$ is $D_t$ with column/row exchanged) though $U_t$ may not
remain upper triangular, uniform, and $R_t$ may not remain reduced, which we fix
by appropriate updates for different cases. Since we do not essentially
change any block by transposition, the map $h$ is an isomorphism in all cases.
However, for the block $B$, we may have a change of basis for $H(B)$ which
is reflected in the change of representatives for the bars associated to columns
$c_j$ and/or $c_{j+1}$. Unlike split and merge, 
columns associated with the bars may change as the columns are exchanged.

In the following, all case labels refer to the status
of the columns in $R_t$ prior to the exchange.

Case(i): Both $c_j$ and $c_{j+1}$ are homogeneous.\\
Case(i.1): The column $c_j$ has been added to $c_{j+1}$ in $R_t$, which 
after the column/row exchanges, means that $c_{j+1}$ has been added to $c_j$. This apparent right-to-left addition makes $U_t$ lose upper triangularity. We cancel the addition by adding $c_{j+1}$ back to $c_j$. 
The corresponding addition of $u_{j+1}$ to $u_j$ fixes $U_t$ to be upper triangular.
Next, we reduce $R_t$. 

If $R_t$ is already reduced at this point, no other updates are necessary. If not,
we claim that the only column that may not be reduced is $c_{j+1}$ because it conflicts with
$c_j$. To see this, first observe that column exchanges and column additions
so far have affected only $c_j$ and $c_{j+1}$. So, any other column
can lose its reduced status only by the row exchange. However, no column
could have $r_j$ or $r_{j+1}$ as its pivot row because that would violate
Proposition~\ref{prop:pivot-hom}. Exchanging non-pivot entries in a column
does not affect its pivot and hence its reduced status.

To reduce $c_{j+1}$,
we add $c_j$ to $c_{j+1}$ which, in effect, makes $c_{j+1}$ 
the original $c_{j+1}$ before exchange and thus reduced. To reflect this
addition in $U_t$, we also add column
$u_j$ to $u_{j+1}$ in $U_t$. At this point, $R_{t+1}\leftarrow R_t$ is reduced, $U_{t+1}\leftarrow U_t$ is upper
triangular, and $R_{t+1}=D_{t+1}U_{t+1}$.


\noindent Case(i.2): The column $c_j$ has not been added to $c_{j+1}$ in $R_t$. In this case, after the column/row exchange, we already have the $R_t$ block reduced, so no other
updates are necessary.

In both cases of Case(i.1) and Case(i.2), no column
changes its homogeneity and targetability status. So, no qualitative update is necessary on the barcode except that all active bars at $t$ extend to index $t+1$ while preserving their column associations and representatives. Indeed, the isomorphism $h$ in these cases does not change basis.\\

Case(ii): $c_j$ is homogeneous but $c_{j+1}$ is non-homogeneous.\\
Case(ii.1): $c_j$ has been added to $c_{j+1}$ in $R_t$. After the column/row exchange, we perform column updates as in Case(i.1) to cancel the (apparent) addition of
$c_{j+1}$ to $c_j$ and then reduce $c_{j+1}$. One can check that, in this case,
after these column updates, $c_j$ and $c_{j+1}$ remain homogeneous
and non-homogeneous, respectively, while $c_{j+1}$ remains the same as before the exchange.

As pointed out in Case(i.1), for row exchange, only columns $c_k$ that cannot possibly remain reduced are columns that have $r_j$ and $r_{j+1}$ as pivot rows.
No such column could have $r_{j}$ as a pivot row because $c_{j}$ was
homogeneous. So, assume that a column $c_k$ had $r_{j+1}$ as a pivot row, and therefore $c_{j+1}$ was targeted. After the exchange $r_j\leftrightarrow r_{j+1}$, $r_{j}$ cannot become a pivot row
of $c_k$ because $c_{j}$ remains homogeneous, creating a contradiction
to Proposition~\ref{prop:pivot-hom}. So, after the row exchange
$r_{j+1}$ has to be the pivot row of $c_k$ making $c_{j+1}$ targeted.
Other columns $c_k$ that did not have $r_{j+1}$ as pivot are not affected
by row exchanges and remain reduced because they could not have $r_j$ as pivot row either. This implies that the homogeneity and targetability of columns
in $R_{t+1}$ remain the same as in $R_t$. In particular, the
non-homogeneous column $c_{j+1}$ remains non-targeted if and only if it was so in
$R_t$. Furthermore, since $c_{j+1}$ does not change from $R_t$ to $R_{t+1}$,
every column $u_\ell$ remains unchanged from $U_t$ to $U_{t+1}$ (modulo the row exchange) where $c_\ell$
was non-homogeneous and non-targeted. This means that the isomorphism $h$ does not
change basis, which implies no qualitative change in the barcode, 
the bars that were alive at $t$ still continue to be alive at $t+1$ with the same column
association and representatives.\\

Case(ii.2): $c_j$ has not been added to $c_{j+1}$ in $R_t$. In this case, no column
updates are necessary after the column/row exchange to make $U_t$ upper triangular.
The matrix $R_t$ remains reduced even as the column $c_j$ and $c_{j+1}$
exchange their homogeneity, that is, $c_j$ becomes non-homogeneous
and $c_{j+1}$ becomes homogeneous. Similar analysis to Case(ii.1) shows that
$c_j$ becomes targeted if and only if $c_{j+1}$ was targeted before $R_t$ was changed.
So, we can draw the same conclusion about the bars and their association with columns
as in Case(ii.1) except that the bar (if any) associated with the column $c_{j+1}$ in $R_t$ becomes associated with the column $c_j$ in $R_{t+1}$.
The representative of the
bar remains the same before and after the exchange, since
we have $u_{j}\leftarrow u_{j+1}$ with the exchange. The isomorphism
$h$ in this case also preserves the basis.\\

Case(iii): $c_j$ is non-homogeneous but $c_{j+1}$ is homogeneous in $R_t$.
In this case, $c_j$ could not have been added to $c_{j+1}$. Therefore, after the exchanges of columns/rows, as in Case(ii.2), no column
updates are necessary to make $U_t$ upper triangular and $R_t$ reduced.
Since new $c_j$ was actually $c_{j+1}$ before and new $c_{j+1}$ was actually
$c_j$ before, the
columns $c_j$ and $c_{j+1}$ exchange their homogeneity/targetability status.
We can draw the same conclusion about the bars and their representatives as in Case(ii.2).\\

Case(iv): Both $c_j$ and $c_{j+1}$ are non-homogeneous. Since $c_j$ could not have been
added to $c_{j+1}$ in $R_t$, the column/row exchange does not need
any other column updates to keep these two columns reduced, and to keep $U_t$ upper triangular and uniform. However, depending on the targetability of $c_j$ and $c_{j+1}$ before the exchange, 
the exchange of rows may trigger some column additions and 
adjustment in bar association to columns. 

First, observe that any column $c_k$ that was non-homogeneous in $R_t$ remains
non-homogeneous after row exchange $r_j\leftrightarrow r_{j+1}$. 
So, no other update is necessary for these columns. For homogeneous
columns, we have
one of the three cases before the exchange.\\
Case(iv.1): $c_k$ is homogeneous and targets
$c_j$ where $c_{j+1}$ is not targeted. After the exchange
$c_k$ remains reduced and targets $c_{j+1}$. Other homogeneous columns remain reduced
because they cannot have $r_j$ or $r_{j+1}$ as a pivot row.
The bar $b$ associated with
$c_{j+1}$ in $R_t$ gets associated with $c_j$ after the exchange in $R_{t+1}$ without
change in representative. \\
Case(iv.2) $c_k$ is homogeneous and targets $c_{j+1}$ where $c_j$ is not targeted. 
The row exchange either makes $c_k$ target $c_j$ if $R_t[j,k]=0$ or target $c_{j+1}$ if
$R_t[j,k]=1$. In both cases $c_k$ and all other homogeneous columns $c_{k'}$, $k'\not=k$, remain reduced
because any such homogeneous column cannot have $r_j$ or $r_{j+1}$ as pivot. 
In the first case,
the bar $b$ associated with $c_j$ gets associated with $c_{j+1}$ after the exchange, 
thus without any change in representative.

In the second case
$b$ remains associated with $c_{j}$ after the exchange.
Observe that in this case 
the representative chain for the bar $b$ changes before and after the transposition.
It changes from $u_j$ to $u_{j+1}$
due to the exchange $c_j\leftrightarrow c_{j+1}$. The isomorphism $h$ in this
case changes the basis, namely, it maps the basis element $u_j$ before the exchange 
to $u_{j+1}$ after the exchange while
preserving all others.
\cancel{
Denoting $u^t_{j}:=u_j$ and $u^{t+1}_{j+1}:=u_{j+1}$,
we need that their classes satisfy $[u_j^t]=[u_{j+1}^{i+1}]$ in the block they belong to.
We modify $u_{j+1}$ to achieve this. The column $c_k$ targets $u_{j+1}$ and contains
both $\sigma_j$ and $\sigma_{j+1}$ (both $R[j,k]=1$ and $R[j+1,k]=1$) where $c_j=c_{\sigma_j}$ and $c_{j+1}=c_{\sigma_{j+1}}$. We have $c_k=\partial u_k$. Consider the chain
$u_{j+1}':=c_k+u_j$ which does not contain $\sigma_j$ but contains $\sigma_{j+1}$.
We also have $[u_j]=[u_{j+1}'+c_k]=[u_{j+1}']+[\partial u_k]=[u_{j+1}']$.
We replace $u_{j+1}\leftarrow u_{j+1}'$ and $c_{j+1}\leftarrow \partial u_{j+1}'$.
}

Case(iv.3) $c_k$ is homogeneous and targets, say $c_j$, where another homogeneous column $c_{k'}$ targets $c_{j+1}$. Assume that $k<k'$. If $R_t[j,k']=1$,
then the exchange of rows makes 
$r_{j+1}$ the pivot row of both $c_k$ and $c_{k'}$. We add
$c_k$ to $c_{k'}$ making $r_j$ the pivot row of $c_{k'}$. This means that 
$c_k$ and $c_{k'}$ exchange pivots after column addition and both get reduced. The same happens in the case $R_t[j,k']= 0$. All other homogeneous
columns $c_\ell$ where $\ell\neq k,k'$ remain reduced because they cannot have
$r_j$ or $r_{j+1}$ as a pivot row. We update $U_t$ to reflect any addition
in $R_t$. The case $k>k'$ is similar. In these cases, no changes are needed for the
barcode because there is no change in homogeneity and targetability of columns in $R_{t+1}$.
In other words, the isomorphism $h$ does not change the basis.

\section{Some aspects of Conley-Morse persistence modules}
\subsection{Simplex-wise expansion}
\label{sec:elementary-expansion}
\begin{theorem}\label{thm:elementary-expansion}
        Every filtration $\zzBD$ of block partitions can be converted into an elementary filtration of block partitions $\zzBD'$ where the
        Conley-Morse persistence barcode for $\zzBD$ can be read from the Conley-Morse 
        persistence barcode of $\zzBD'$.
\end{theorem}
We prove Theorem~\ref{thm:elementary-expansion} by showing a single step of the expansion of a zigzag filtration of block partitions toward an elementary simplex-wise filtration.
The proof follows by its recursive application.

Let $(B_1, B_2)$ be a split of a block $B$.
Assume that $B_1$ is closed in $B$ and $\sigma$ is a top-dimensional simplex in $B_2$. 
Let us recall the general formula for the split together with the constructed homomorphisms derived in Section~\ref{sec:construction-of-structural-map-h}.
\begin{equation}\label{eq:lefschetz_ar_split_regular-appendix}
    \begin{tikzcd}[row sep=tiny, column sep=2.5cm]
        & H(\bl_2)=V_1\oplus V_2\arrow[ld, "f=0\oplus (j_\ast|_{X_2})^{-1}",sloped]\\
        X_1\oplus X_2 = H(\bl) & \\
        & H(\bl_1)=W_1\oplus W_2\arrow[lu, "g=(i_\ast|_{W_1})\oplus 0",swap, sloped]
    \end{tikzcd}
\end{equation}
We expand the split by extracting $\sigma$ from $B_2$ as presented below, 
    where each pair $f_l$, $g_l$, for $l\in\{1,2,3\}$ arises from a local split:
\begin{equation}\label{diag:simplex-wise-expansion}
    \begin{tikzcd}
        & & H(\{\sigma\})\arrow[ld, "\id_2", sloped]\arrow[r, "f_3"] & H(B_2)\eqqcolon V \\
        & H(\{\sigma\})\arrow[ld, "f_1", sloped] & H(B_2\setminus\{\sigma\})\arrow[ld, "f_2", sloped]\arrow[ru, "g_3", sloped] & \\
        X\coloneqq H(B) & H(B\setminus\{\sigma\})\arrow[l, "g_1", swap] & H(B_1)\arrow[l, "g_2", swap]\arrow[r, "\id_3"] & H(B_1)\eqqcolon W
    \end{tikzcd}
\end{equation}
We show direct correspondence between maps in diagrams~\eqref{eq:lefschetz_ar_split_regular-appendix} and~\eqref{diag:simplex-wise-expansion}.

\begin{proposition}\label{prop:single-step-of-the-expansion}
    Let $(B_1, B_2)$ be a split of a block $B$.
    Assume that $B_1$ is closed in $B$ and $\sigma$ is a top-dimensional simplex in $B_2$. 
    Consider maps as in diagrams~\eqref{eq:lefschetz_ar_split_regular-appendix} and~\eqref{diag:simplex-wise-expansion}.
    We have $g= g_1\circ g_2$ and $f=(f_1\circ f_3^{-1}) + (g_1\circ f_2\circ g_3^{-1})$,
    where $f:H(B_2)\rightarrow H(B)$ and $g:H(B_1)\rightarrow H(B)$ are defined in Section~\ref{sec:construction-of-structural-map-h}.
\end{proposition}

It follows that we can easily retrieve the barcode for~\eqref{eq:lefschetz_ar_split_regular-appendix}
    by computing the barcode for~\eqref{diag:simplex-wise-expansion}, 
    one only needs to discard spurious bars generated by the expansion of the filtration, 
    similarly to the standard persistence computations.
The expansion step can be repeatedly applied to $\zzBD$ until we obtain an elementary filtration $\zzBD'$.


\begin{proof}[Proof of Proposition~\ref{prop:single-step-of-the-expansion}]
    First, we will show the equality for $f$. 
    Let $u\in\Zgroup(B_2)$ be such that $[u]$ is a basis element in $V$ consistent with the choice of basis $S_{V_1}\cup S_{V_2}$ described in Section~\ref{sec:construction-of-structural-map-h}.
    
    Notice that $\{\sigma\}$ and $B_2\setminus\{\sigma\}$ form a split of $B_2$, 
        and we follow the same construction to define $f_3$ and $g_3$.
    We can choose bases for this elementary split compatible with $S_{V_1}\cup S_{V_2}$.
    In particular, we obtain a subdivision of $S_{V_2}$ into $S_{V_2}^{f_3}$ and $S_{V_2}^{g_3}$ associated with $f_3$ and $g_3$, respectively.
    Similarly, for the split of $B$ into $\{\sigma\}$ and $B\setminus\{\sigma\}$ we subdivide $S_{X_1}$ into $S_{X_1}^{f_1}$ and $S_{X_1}^{g_1}$.
    
    \textbf{Case 1.} 
        Assume that $f([u])\neq 0$.
        Then, by construction $[u]_{V}\in S_{V_2}$.
        Additionally, by Proposition~\ref{prop:reversing-the-split-map}.\ref{it:reversing-the-split-map-B2}, 
            there exists $e\in\Cgroup(B_1)$ such that 
            $d\coloneqq u+e\in\Zgroup(B)$ and $f([u]_{V})=[d]_X$.
        We have two subcases to consider:
        
    \textbf{Case 1.1} 
        $\sigma\not\in u$.
        Since $\sigma$ is a top dimensional simplex, 
            it follows that $u$ is still a cycle in $B_2\setminus\{\sigma\}$,
            $g_3([u])=[u]$ and therefore $[u]\in S_{V_2}^{g_3}$.
        Note that the chosen $d$ and $e$ satisfies the condition of Proposition~\ref{prop:reversing-the-split-map}.\ref{it:reversing-the-split-map-B2} for the $B\setminus\{\sigma\}=B_1\cup(B_2\setminus\{\sigma\})$ split.
        Thus, we deduce that $f_2([u])=[d]$.
        By the assumption we know $[d]$ is a basis element of $X$,  
            hence $g_1([d])=[d]$ and $f([u])=(g_1\circ f_2\circ g_3^{-1})([u])=[d]$.
        
    \textbf{Case 1.2}
        $\sigma\in u$.
        Then it follows that $[u]\in S_{V_2}^{f_3}$.
        By construction $f_3^{-1}([u])=[\sigma]$ and $g_3^{-1}([u])=0$.
        Now, let $v\in\Cgroup(B)$ be such that $u=\sigma+v$, in particular, 
            $d=\sigma+v+e\in\Zgroup(B)$.
        Therefore, by Proposition~\ref{prop:reversing-the-split-map}.\ref{it:reversing-the-split-map-B2} we have $f_1([\sigma])=[d]$.
        Thus, $f([u])=(f_1\circ f_3^{-1})([u])$.
        
    \textbf{Case 2} Assume that $f([u])=0$, then by construction $[u]_V\in S_{V_1}$.
       By Proposition~\ref{prop:reversing-the-split-map}.\ref{it:reversing-the-split-map-B2} there is no $d$ as in \textbf{Case 1}.
       
    \textbf{Case 2.1} 
        Suppose that $\sigma\not\in u$.
        Then, by similar argument as in \textbf{Case 1.1}, we have $g_3^{-1}([u])=[u]$.
        However, since $d$ does not exists, then again by 
            Proposition~\ref{prop:reversing-the-split-map}.\ref{it:reversing-the-split-map-B2}
            we have $f_2([u])=0$.
        Therefore, $(g_1\circ f_2\circ g_3^{-1})([u])=0$.
        
    \textbf{Case 2.2} 
        Assume that $\sigma\in u$.
        As in \textbf{Case 1.2} we show that $f_3^{-1}([u])=[\sigma]$, 
            and again, since $d$ doesn't exists we have $f_1([\sigma])=0$.
        This finishes the proof that $f=(f_1\circ f_3^{-1}) + (g_1\circ f_2\circ g_3^{-1})$.

    The equivalence for $g$ follows from the fact that $g$, $g_1$ and $g_2$ all are directly induced by inclusion.
    Therefore, for any $u\in\Zgroup(B_1)$ we have $(g_1\circ g_2)([u]_W)=g_1([u])=[u]_X = g([u]_W)$.
    \qed
\end{proof}

\subsection{Connecting the splits: an example}
    \label{sec:appendix_connecting_the_splits}

In Section~\ref{sec:construction-of-structural-map-h} 
    we split the spaces $X$, $W$, and $V$ 
    and choose the bases for them to fix the structural maps $f$ and $g$ for the triple in the Conley-Morse persistence module.
However, it may happen that the space $H(\bl)$ of a block $B$ splits differently in the forward and the backward direction of the filtration.
Since the structural maps $f$, $g$ are defined by the local choice of bases
    we may end up with incompatible structural maps in the forward and the backward direction.

A situation like this emerges, for instance, when a block $B$ is partitioned differently when going forward and backward.
Filtration $\zzBD:\cB_0\inscr\cB_1\ovscr\cB_2$ in Figure~\ref{fig:base-change} exemplifies such possibility.
Block $\bl_{1,1}$ splits differently when going backward and forward.
In step $\cB_0\inscr\cB_1$ edge $\bl_{2,0}=\{ab\}$ merges with $\bl_{1,0}$ forming $\bl_{1,1}$,
    whereas in the second step, edge $bc$ splits from $\bl_{1,1}$.
Let us write $\hat{X}_1\oplus \hat{X}_2$ for the backward split of $H(\bl_{1,1})$, $X_1\oplus X_2$ for its forward split, and $\hat{f}$, $\hat{g}$, $f$, and $g$ for the corresponding structural maps.
To complete the picture let us introduce isomorphism $\gamma$ connecting the two splits.
We summarize this with the diagram below:

\begin{equation}\label{eq:lefschetz_ar_split_regular_mirrored}
\begin{tikzcd}[row sep=small, column sep=2.7cm]
    \hat{V}_1\oplus \hat{V}_2 
        \arrow[rd, "\hat{f}=0\oplus (j_\ast|_{\hat{X}_2})^{-1}", sloped]
        & & &
        V_1\oplus V_2
            \arrow[ld, "f=0\oplus (j_\ast|_{X_2})^{-1}", sloped] \\
    &
    \hat{X}_1\oplus \hat{X}_2 \arrow[r, "\gamma", shorten >=6pt, shorten <=6pt]
    &
    X_1\oplus X_2
    \\
    \hat{W}_1\oplus \hat{W}_2
        \arrow[ru, "\hat{g}=(i_\ast|_{\hat{W}_1})\oplus 0", swap, sloped]
        & & &
        W_1\oplus W_2
            \arrow[lu, "g=(i_\ast|_{W_1})\oplus 0", swap, sloped]
\end{tikzcd}
\end{equation}

We emphasize that $\gamma$, that is, the change of basis, is not simply a homomorphism induced by the identity map.
The decomposition theorem~\cite[Theorem~7.6]{CMbarcodes2025}, and the gentle algebra behind this theorem,
    constructs $\gamma:\hat{X}_1\oplus \hat{X}_2\rightarrow X_1\oplus X_2$ implicitly in such a way that the fixed basis elements of $\hat{X}$ are bijectively mapped into basis elements of $X$ (note that, e.g., element of $\hat{X}_1$ can be mapped into $X_2$).
In our algorithm, $\gamma$ is also constructed implicitly within the transposition step, which does not modify the space and its partition, but changes the representatives of the basis when needed.

Consider the bases in Figure~\ref{fig:base-change} (top).
Let the green edges $\alpha_1$ and $\alpha_2$ be representatives for the bases chosen for the first split and the orange edges $\beta_1$ and $\beta_2$ be representatives for the second split. 
This choice of bases is unambiguous, 
    the homomorphism induced by the identity gives a correct bijection between the bases elements,
    that is,
    $\gamma([\alpha_1])=[\beta_2]$ and $\gamma([\alpha_2])=[\beta_1]$.
    
On the other hand, if we choose representatives as in the bottom panel of Figure~\ref{fig:base-change} 
    then the identity induced homomorphism would map 
        $[\alpha_1]$ into $[\beta_1 + \beta_2]$ and $[\alpha_2]$ into $[\beta_2]$.
However, by the decomposition theorem, 
    there exists a consistent way of connecting the bases that leads to the string decomposition.
As described in Section~\ref{sec:split-and-merge} merging the edge $ab$ into block $\bl_{1,0}$ puts $ab$ at the beginning of the section in the boundary matrix corresponding to $\bl_{1,1}$.
This follows from the fact that $\bl_{2,0}$ is lower than $\bl_{1,0}$ in the poset on elements of $\cB_0$ (see Section~\ref{subsec:mvf}).
For the same reason, to extract $bc$ from $\bl_{1,1}$, 
    we need to move $bc$ to the beginning of the $\bl_{1,1}$.
Therefore, at least one transposition step is required between the merge and the split, that is, the transposition of $ab$ and $bc$.
If we follow the transposition step for this example we would notice that $\alpha_1$ continues to $\beta_2$ and $\alpha_2$ to $\beta_1$.

\begin{figure}[hbt]
    \centering
    \includegraphics[width=0.8\linewidth]{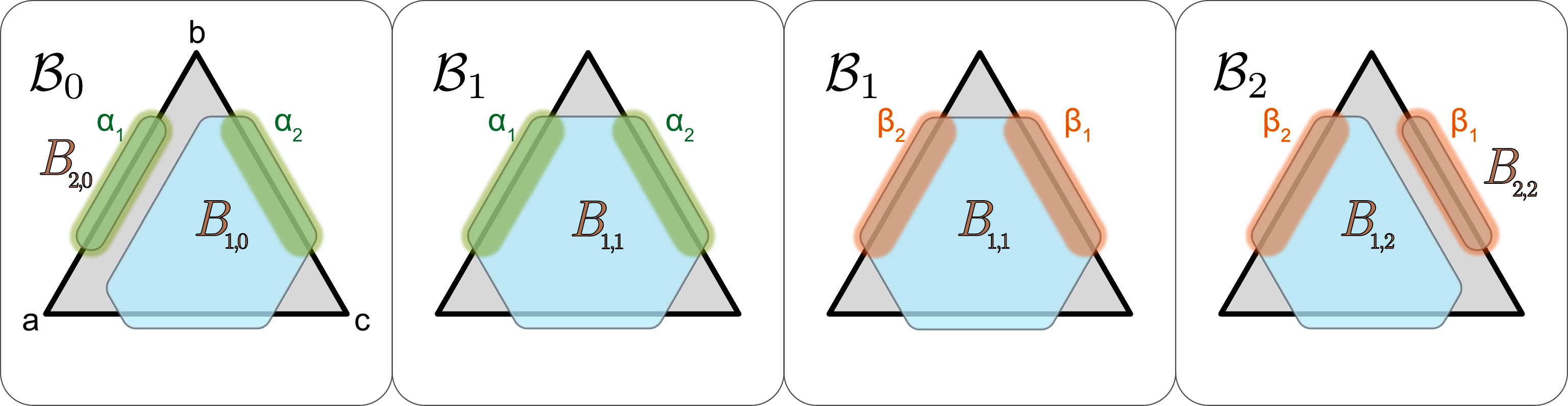}
    \vspace{0.2cm}
    
    \includegraphics[width=0.8\linewidth]{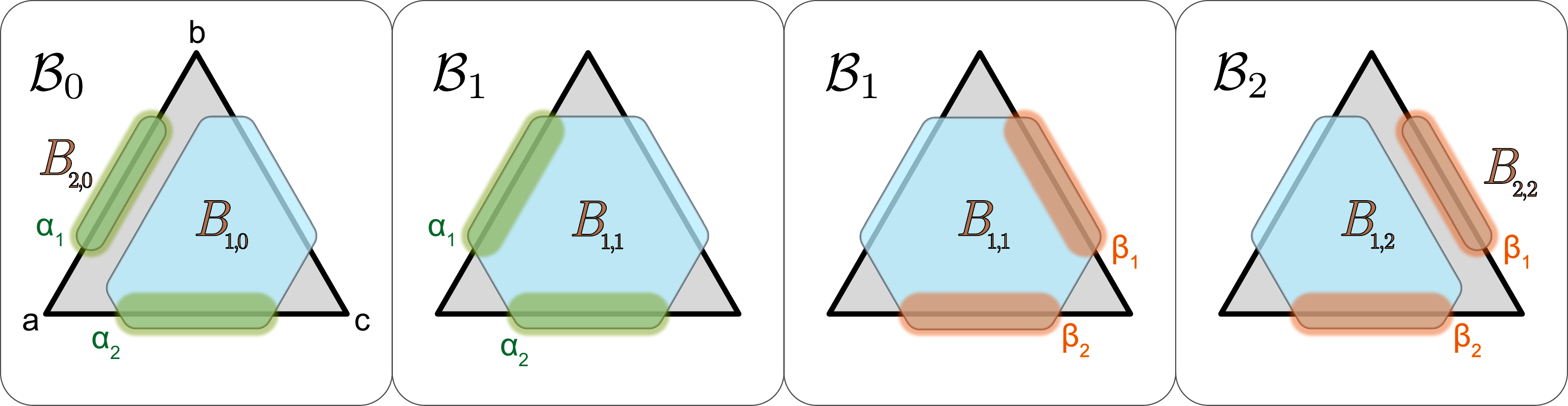}
    \caption{
        A filtration of block partitions $\cB_0\inscr\cB_1\ovscr\cB_2$.
        The top row shows a case when the chosen representatives $\alpha_1$, $\alpha_2$, and $\beta_1$, $\beta_2$ are compatible across the splits.
        On the other hand, the bottom row presents an incompatible choice of representatives.
        }
    \label{fig:base-change}
\end{figure}

\section{Concluding remarks} The result in this paper makes
Conley-Morse
persistence barcode computation practical for elementary operations, and
is a step toward building
an usable software based on it.
Notice that a different linear order among the blocks can be obtained
from the current order by moving an entire block to its right or left.
This movement can be simulated by multiple Split, Merge and Transposition operations. 
To expedite this in practice, we may introduce an update called
Shuffle which does not use Split, Merge and Transposition.

\textbf{Shuffle}:
Suppose that a block $B$ is moved left. Then, we reorder the columns and the rows 
of both $R_t$ and $U_t$. We claim that $R_t$ remains block reduced and no other
updates are required for the barcode.

The blocks that precede $B$ in $R_{t}$ are not affected by the movement of $B$ because the columns of these blocks cannot have non-zero entries in rows that belong to block $B$.
This claim follows from Proposition 4.4 in~\cite{DeHaLi2026} and the fact that $B$ moves past only blocks that are incompatible
with $B$. We note that, to apply
this result, one needs that all columns are reduced only by homogeneous columns from
left. This is ensured by 
maintaining the matrix $U_t$ to be uniform by the update algorithm.  

 The block $B$ itself also does not change because the orders of its rows and columns do not change and the rows that have non-zero entries do not change order after the movement because they cannot belong to blocks succeeding $B$ after the movement.

Then, the only change happens for blocks that succeed $B$ in $R_t$ before the movement. Let $c$ be any reduced column in such a block. If $c$ is homogeneous in $R_t$, its pivot row remains the same because that column belongs to the same block. So, $c$ remains homogeneous
targeting the same column. So, except for possible reordering of its non-pivot
rows, no change is necessary for $c$. If $c$ is non-homogeneous in $R_t$, its rows including the pivot row may be reordered, but that does not affect anything else. 
In both cases, no change is necessary for the column corresponding to $c$ in $U_t$ because no extra additions are made to the column $c$. 

Over all cases, we see that $R_{t+1}$ remains reduced with $R_{t+1}=D_{t+1}U_{t+1}$ and
$U_{t+1}$ remaining uniform. Also, since $U_t$ does not change except permutation of its rows and columns, all bars continue without any change in their representatives.

\section*{Acknowledgment}
 T.D. acknowledges the support of NSF funds CCF-2437030 and DMS-2301360. M.L. acknowledges support from the European Union’s Horizon 2020 research and innovation programme under the Marie Skło\-dow\-ska-Curie Grant Agreement No.~101034413.

\bibliography{bibliography}

@inproceedings{milosavljevic2011zigzag,
  title={Zigzag persistent homology in matrix multiplication time},
  author={Milosavljevi{\'c}, Nikola and Morozov, Dmitriy and Skraba, Primoz},
  booktitle={Proceedings of the Twenty-Seventh Annual Symposium on Computational Geometry},
  pages={216--225},
  year={2011}
}

@inbook{Bubenik2024,
	address = {Cham},
	author = {Bubenik, Peter and Catanzaro, Michael J.},
	pages = {55--79},
	publisher = {Springer Nature Switzerland},
	title = {Multiparameter Persistent Homology via Generalized Morse Theory},
	year = {2024}}

@inproceedings{CaSiMo2009-zigzag,
	address = {New York, NY, USA},
	author = {Carlsson, Gunnar and de Silva, Vin and Morozov, Dmitriy},
	booktitle = {Proceedings of the Twenty-Fifth Annual Symposium on Computational Geometry},
	pages = {247--256},
	publisher = {Association for Computing Machinery},
	series = {SCG '09},
	title = {Zigzag persistent homology and real-valued functions},
	year = {2009}}

@book{DW22,
    author  =   {Tamal K. Dey and Yusu Wang},
    title   =   {Computational Topology for Data Analysis},
    publisher= {Cambridge University Press},
    year    =   {2022}
}

@article{DeHaLi2026,
	author = {Dey, Tamal K. and Haas, Andrew and Lipi\'{n}ski, Micha\l{}},
	doi = {10.1137/25M1739406},
	journal = {SIAM Journal on Applied Dynamical Systems},
	number = {1},
	pages = {108-130},
	title = {Computing a Connection Matrix and Persistence Efficiently from a Morse Decomposition},
	url = {https://doi.org/10.1137/25M1739406},
	volume = {25},
	year = {2026}}

@inproceedings{DeLiMrSl2022,
	author = {Dey, Tamal K. and Lipi\'nski, Micha\l{} and Mrozek, Marian and Slechta, Ryan},
	booktitle = {38th Symposium on Computational Geometry},
	title = {{Tracking dynamical features via continuation and persistence}},
	year = {2022}}

@book{Conley1978,
	address = {Providence, R.I.},
	author = {Charles Conley},
	publisher = {American Mathematical Society},
	series = {CBMS Regional Conference Series in Mathematics},
	title = {Isolated Invariant Sets and the {M}orse Index},
	volume = {38},
	year = {1978}}

@article{CMbarcodes2025,
	author = {Dey, Tamal K. and Lipi{\'n}ski, Micha{\l} and Soriano-Trigueros, Manuel},
	date = {2026/08/04},
	doi = {10.1007/s10208-026-09766-6},
	id = {Dey2026},
	isbn = {1615-3383},
	journal = {Foundations of Computational Mathematics},
	title = {Conley--{M}orse Persistence Barcode: A Homological Signature of Combinatorial Bifurcations},
	url = {https://doi.org/10.1007/s10208-026-09766-6},
	year = {2026}}

@article{Forman1998b,
	author = {Forman, Robin},
	journal = {Mathematische Zeitschrift},
	number = {4},
	pages = {629--681},
	title = {Combinatorial vector fields and dynamical systems},
	volume = {228},
	year = {1998}}

@article{Franzosa1988,
	author = {Robert D. Franzosa},
	journal = {Transactions of the American Mathematical Society},
	number = {2},
	pages = {781--803},
	title = {The Continuation Theory for {M}orse Decompositions and Connection Matrices},
	volume = {310},
	year = {1988}}

@article{Kim:2021wx,
	author = {Kim, Woojin and M{\'e}moli, Facundo},
	journal = {Discrete \& Computational Geometry},
	number = {3},
	pages = {831--875},
	title = {Spatiotemporal Persistent Homology for Dynamic Metric Spaces},
	volume = {66},
	year = {2021}}

@article{King2017,
	author = {Henry King and Kevin Knudson and Ne{\v z}a {Mramor Kosta}},
	journal = {Journal of Symbolic Computation},
	pages = {41-60},
	title = {Birth and death in discrete {M}orse theory},
	volume = {78},
	year = {2017}}

@article{Hotz2023,
	author = {Yan, Lin and Masood, Talha Bin and Rasheed, Farhan and Hotz, Ingrid and Wang, Bei},
	journal = {IEEE Transactions on Visualization \& Computer Graphics},
	month = aug,
	number = {08},
	pages = {3489-3506},
	title = {{ Geometry-Aware Merge Tree Comparisons for Time-Varying Data With Interleaving Distances }},
	volume = {29},
	year = {2023}}

@article{LKMW2022,
	title = {Conley-{Morse}-{Forman} theory for generalized combinatorial multivector fields on finite topological spaces},
	volume = {7},
	issn = {2367-1726, 2367-1734},
	doi = {10.1007/s41468-022-00102-9},
	number = {2},
	urldate = {2024-01-21},
	journal = {Journal of Applied and Computational Topology},
	author = {Lipiński, Michał and Kubica, Jacek and Mrozek, Marian and Wanner, Thomas},
	year = {2023},
	pages = {139--184},
}

@misc{MeLiCh2026,
      title={MS-COOT: Comparing Morse-Smale Complexes with Co-Optimal Transport}, 
      author={Guangyu Meng and Mingzhe Li and Erin Wolf Chambers},
      year={2026},
      eprint={2606.08258},
      archivePrefix={arXiv},
      primaryClass={cs.GR},
	howpublished = {arXiv:2606.08258},
	doi = {10.48550/arXiv.2606.08258}
}

@incollection{MischMro_Conley_2002,
	title = {The {C}onley Index},
	volume = {2},
	series = {Handbook of Dynamical Systems},
	pages = {393--460},
	booktitle = {Handbook of Dynamical Systems},
	publisher = {Elsevier Science},
	author = {Mischaikow, Konstantin and Mrozek, Marian},
	editor = {Fiedler, Bernold},
	year = {2002},
	doi = {10.1016/S1874-575X(02)80030-3},
	note = {{ISSN}: 1874-575X},
}

@article{Mrozek2017,
	author = {Marian Mrozek},
	journal = {Foundations of Computational Mathematics},
	number = {6},
	pages = {1585--1633},
	title = {{C}onley--{M}orse--{F}orman Theory for Combinatorial Multivector Fields on {L}efschetz Complexes},
	volume = {17},
	year = {2017}}

@book{MroWan2025,
	author = {Mrozek, Marian and Wanner, Thomas},
	edition = {1},
	month = {July},
	publisher = {Springer Cham},
	series = {SpringerBriefs in Mathematics},
	title = {Connection Matrices in Combinatorial Topological Dynamics},
	year = {2025},
    doi = {10.1007/978-3-031-87600-4}
}

@article{GuMuKh2022,
	author = {G{\"u}zel, {\.I}smail and Munch, Elizabeth and Khasawneh, Firas A.},
	journal = {Chaos: An Interdisciplinary Journal of Nonlinear Science},
	month = {09},
	number = {9},
	pages = {093111},
	title = {Detecting bifurcations in dynamical systems with CROCKER plots},
	volume = {32},
	year = {2022},
        doi = {10.1063/5.0102421}
}

@article{TyMuKh2020,
	author = {Tymochko, Sarah and Munch, Elizabeth and Khasawneh, Firas A.},
	journal = {Algorithms},
	number = {11},
	title = {Using Zigzag Persistent Homology to Detect Hopf Bifurcations in Dynamical Systems},
	volume = {13},
	year = {2020},
        pages = {1--16},
        doi = {10.3390/a13110278}}

@inproceedings{CohEdeMor2006,
	address = {New York, NY, USA},
	author = {Cohen-Steiner, David and Edelsbrunner, Herbert and Morozov, Dmitriy},
	booktitle = {Proceedings of the 26th Annual Symposium on Computational Geometry},
	pages = {119--126},
	publisher = {Association for Computing Machinery},
	series = {SCG '06},
	title = {Vines and vineyards by updating persistence in linear time},
	year = {2006}}

@article{MroBat2009,
	author = {Mrozek, Marian and Batko, Bogdan},
	journal = {Discrete \& Computational Geometry},
	number = {1},
	pages = {96--118},
	title = {Coreduction Homology Algorithm},
	volume = {41},
	year = {2009}}

@article{ReininghausHotz2012,
	author = {Reininghaus, Jan and Kasten, Jens and Weinkauf, Tino and Hotz, Ingrid},
	journal = {IEEE Transactions on Visualization and Computer Graphics},
	number = {9},
	pages = {1563-1573},
	title = {Efficient Computation of Combinatorial Feature Flow Fields},
	volume = {18},
	year = {2012}}

@incollection{EdelsHarer2008,
	author = {Edelsbrunner, Herbert and Harer, John},
	booktitle = {Surveys on Discrete and Computational Geometry: Twenty Years Later},
	pages = {257-282},
	publisher = {American Mathematical Society},
	title = {Persistent homology - a survey},
	year = {2008}}

\appendix
    
\cancel{
\section{Block partition as a filtration}
The following theorems show close relation between block decomposition and filtration of a complex.
\begin{proposition}\label{prop:BP_is_a_filtration}
    Let $\cB$ be a block partition of $\cV$.
    Then for each $p\in\PP$ the set $B_{\leq p}\coloneqq \bigcup_{q\leq p} B_q$ is closed.
    In particular, 
        if $\leq'$ is a linear extension of $\leq$ then family of sets $B_{\leq' p}$ induces a filtration on~$X$.
\end{proposition}

\begin{proposition}\label{prop:filtration_is_a_BP}
    For any filtration of a Lefschetz complex $X$ \michal{do we need it for $p$-skeletons?}:
    \begin{align*}
        \cX: \emptyset=X_{0} \subset X_{1}\subset \ldots \subset X_{n} = X.
    \end{align*}
    The family of all connected components of all level sets of $\cX$ form a block partition for a certain multivector field $\cV$ on~$\cX$.
\end{proposition}

\begin{proposition}
    Let $B_p\in\cB$, where $\cB$ is a block partition of $\cV$.
    Then $(B_{\leq p}, B_{< p})$ is an index pair for $B_p$.
\end{proposition}
}

\cancel{
\section{Preliminary implementation}
\label{sec:preliminary-implementation}

A notable advantage of the matrix update algorithm discussed here is ease of implementation: the elementary operations - [left/right]-[split/merge], transposition, and shuffling - are simple to carry out. Though we believe that many further optimizations are possible, we have implemented the matrix update algorithm as presented in the paper. We intend to continue iterating on this implementation.

Here we use this implementation to create two examples: a small, toy example with clear intuition and associated visuals; and a larger example. Each example consists of a block partitioned complex, an associated block boundary matrix, and a sequence of updates applied to yield a Conley-Morse persistence barcode.

\subsection{Small example}

\begin{figure}[H]
    \centering
    \includegraphics[width=.325\linewidth]{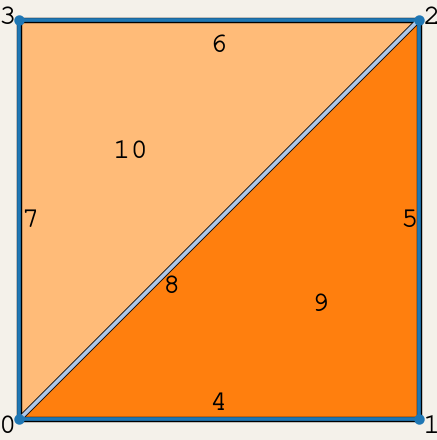}
    \caption{A small block partitioned simplicial complex. Blocks membership is denoted by color, and simplices are each given a unique numeric ID.
    }
    \label{fig:small_initial}
\end{figure}

Figure~\ref{fig:small_initial} shows a block partitioned simplicial complex. Several arbitrary atomic operations are performed:

\[
\resizebox{\textwidth}{!}{%
$
\begin{aligned}
& \xrightarrow{\makebox[4cm][c]{\text{\textbf{[1]} Reduction}}}
\left[
\begin{array}{c|cccccccc|c|c|c}
\textbf{ }  & 0 & 1 & 2 & 3 & 4 & 5 & 6 & 7 & 8 & 9 & 10 \\
\hline
0  &   &   &   &   & 1 &   &   &   & 1 &   &   \\
1  &   &   &   &   & 1 & 1 &   &   &   &   &   \\
2  &   &   &   &   &   & 1 & 1 &   & 1 &   &   \\
3  &   &   &   &   &   &   & 1 &   &   &   &   \\
4  &   &   &   &   &   &   &   &   &   & 1 &   \\
5  &   &   &   &   &   &   &   &   &   & 1 &   \\
6  &   &   &   &   &   &   &   &   &   &   & 1 \\
7  &   &   &   &   &   &   &   &   &   &   & 1 \\
8  &   &   &   &   &   &   &   &   &   & 1 & 1 \\
9  &   &   &   &   &   &   &   &   &   &   &   \\
10 &   &   &   &   &   &   &   &   &   &   &   \\
\end{array}
\right]
\xrightarrow{\makebox[4cm][c]{\text{\textbf{[2]} Right Merge}}}
\left[
\begin{array}{c|ccccccccc|c|c}
\textbf{ }  & 0 & 1 & 2 & 3 & 4 & 5 & 6 & 7 & 8 & 9 & 10 \\
\hline
0  &   &   &   &   & 1 &   &   &   &   &   &   \\
1  &   &   &   &   & 1 & 1 &   &   &   &   &   \\
2  &   &   &   &   &   & 1 & 1 &   &   &   &   \\
3  &   &   &   &   &   &   & 1 &   &   &   &   \\
4  &   &   &   &   &   &   &   &   &   & 1 &   \\
5  &   &   &   &   &   &   &   &   &   & 1 &   \\
6  &   &   &   &   &   &   &   &   &   &   & 1 \\
7  &   &   &   &   &   &   &   &   &   &   & 1 \\
8  &   &   &   &   &   &   &   &   &   & 1 & 1 \\
9  &   &   &   &   &   &   &   &   &   &   &   \\
10 &   &   &   &   &   &   &   &   &   &   &   \\
\end{array}
\right]
\\[1.5ex]
& \xrightarrow{\makebox[4cm][c]{\text{\textbf{[3]} Right Merge}}}
\left[
\begin{array}{c|cccccccccc|c}
\textbf{ }  & 0 & 1 & 2 & 3 & 4 & 5 & 6 & 7 & 8 & 9 & 10 \\
\hline
0  &   &   &   &   & 1 &   &   &   &   &   &   \\
1  &   &   &   &   & 1 & 1 &   &   &   &   &   \\
2  &   &   &   &   &   & 1 & 1 &   &   &   &   \\
3  &   &   &   &   &   &   & 1 &   &   &   &   \\
4  &   &   &   &   &   &   &   &   &   & 1 &   \\
5  &   &   &   &   &   &   &   &   &   & 1 &   \\
6  &   &   &   &   &   &   &   &   &   &   & 1 \\
7  &   &   &   &   &   &   &   &   &   &   & 1 \\
8  &   &   &   &   &   &   &   &   &   & 1 & 1 \\
9  &   &   &   &   &   &   &   &   &   &   &   \\
10 &   &   &   &   &   &   &   &   &   &   &   \\
\end{array}
\right]
\xrightarrow{\makebox[4cm][c]{\text{\textbf{[4]} Column Transpose}}}
\left[
\begin{array}{c|cccccccccc|c}
\textbf{ }  & 0 & 1 & 2 & 3 & 4 & 5 & 6 & 8 & 7 & 9 & 10 \\
\hline
0  &   &   &   &   & 1 &   &   &   &   &   &   \\
1  &   &   &   &   & 1 & 1 &   &   &   &   &   \\
2  &   &   &   &   &   & 1 & 1 &   &   &   &   \\
3  &   &   &   &   &   &   & 1 &   &   &   &   \\
4  &   &   &   &   &   &   &   &   &   & 1 &   \\
5  &   &   &   &   &   &   &   &   &   & 1 &   \\
6  &   &   &   &   &   &   &   &   &   &   & 1 \\
8  &   &   &   &   &   &   &   &   &   & 1 & 1 \\
7  &   &   &   &   &   &   &   &   &   &   & 1 \\
9  &   &   &   &   &   &   &   &   &   &   &   \\
10 &   &   &   &   &   &   &   &   &   &   &   \\
\end{array}
\right]
\end{aligned}
$%
}
\]

We label columns according to associated simplices in Figure~\ref{fig:small_initial}. In step \textbf{[1]} the block boundary matrix is reduced via the block independent reduction algorithm of section~\ref{sec:independent-block-reduction}; in steps \textbf{[2,3]} the leftmost block is merged with its right neighbor; and in step \textbf{[4]} the columns associated with simplices 7 and 8 are transposed, as are the corresponding rows.

Figure~\ref{fig:small_barcode} depicts, on the left, the final partitioned complex resulting from these three update operations. We see that three of the initial five bars persist throughout, and that several bars change in representative cycle and/or associated column.

\begin{figure}[H]
    \centering
    \includegraphics[width=\linewidth]{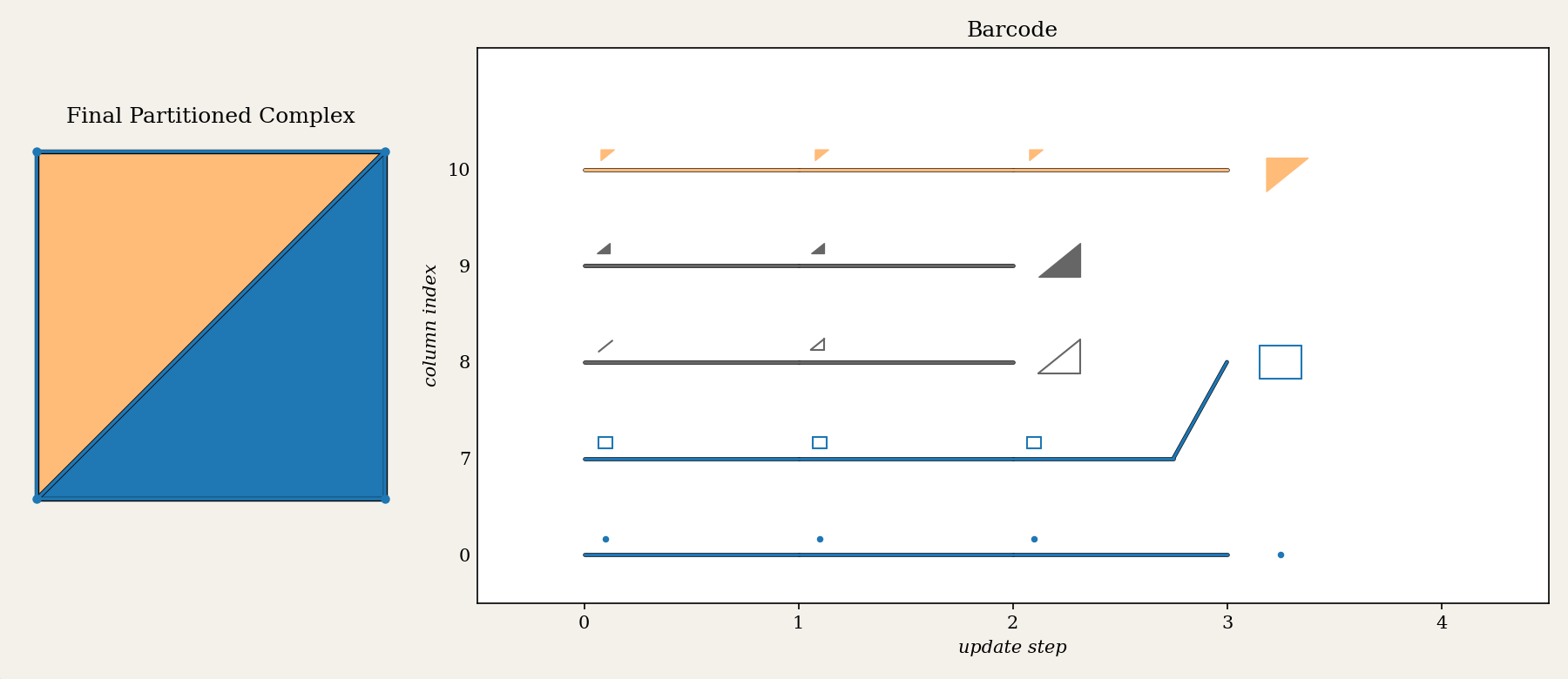}
    \caption{Left: the final block partitioned complex, after three updates have been performed. Right: the barcode associated with these updates. Each bar is associated, at each step, with a representative cycle (shown as a small glyph) and a column index (plotted on the vertical axis).
    }
    \label{fig:small_barcode}
\end{figure}

\subsection{Large example}

\begin{figure}[H]
    \centering
    \includegraphics[width=\linewidth]{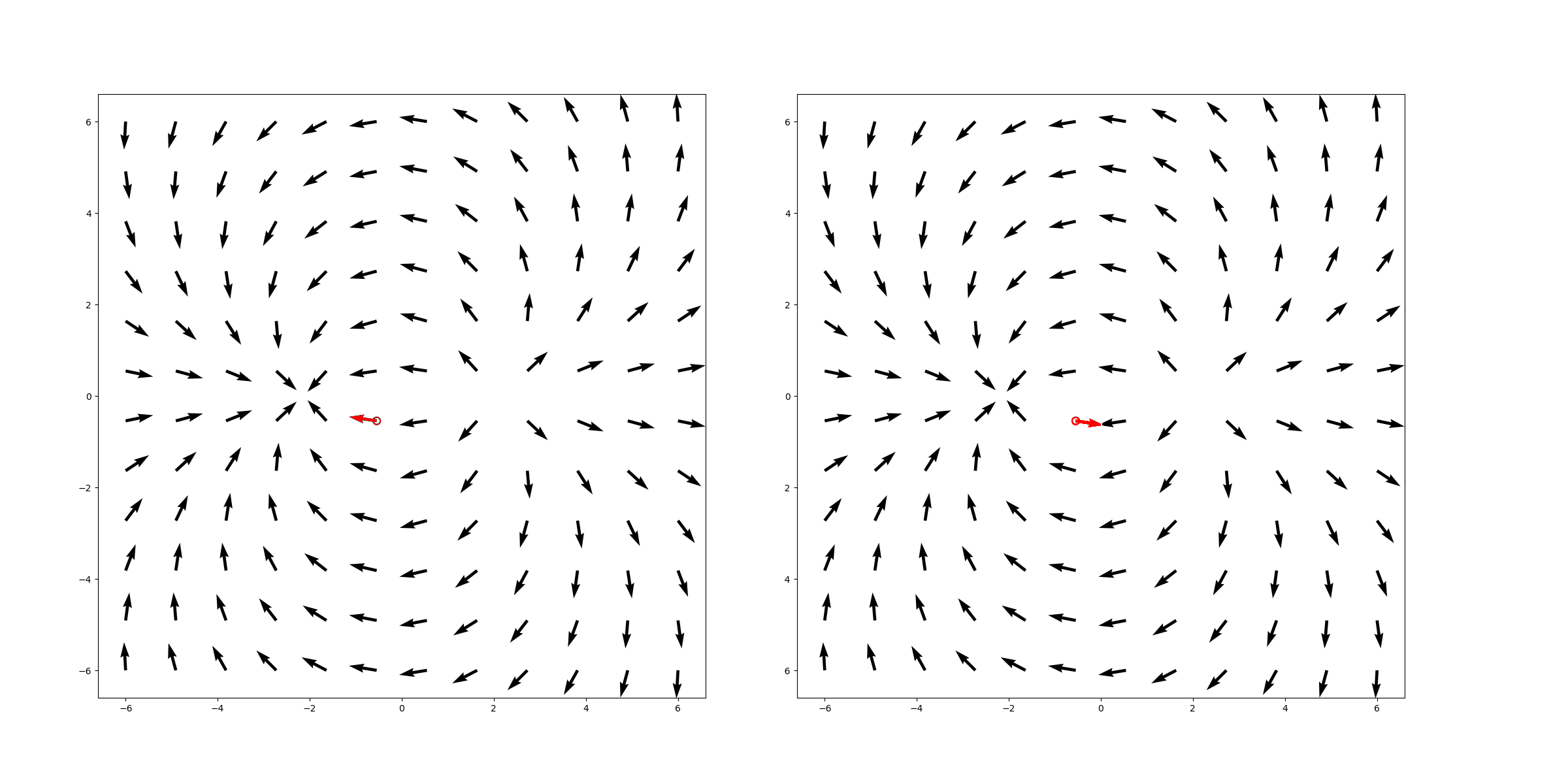}
    \caption{Left: a continuous vector field; Right: the same vector field, subject to a local transformation. Note the reversal of one highlighted arrow.
    }
    \label{fig:vectors_ba}
\end{figure}

Though we will not discuss this process in any detail, Figure~\ref{fig:vectors_ba} illustrates two fields which, when overlayed atop a triangulation in the plane, may be discretized to arrive at two differing multivector fields over the same simplicial complex, and thus two different block partitions of the same complex. We ``interpolate'' between these two block partitions using elementary update operations, yielding a barcode. The complex in question has approximately two-thousand simplices. Note the existence of many short-lived bars.

\begin{figure}[H]
    \centering
    \includegraphics[width=\linewidth]{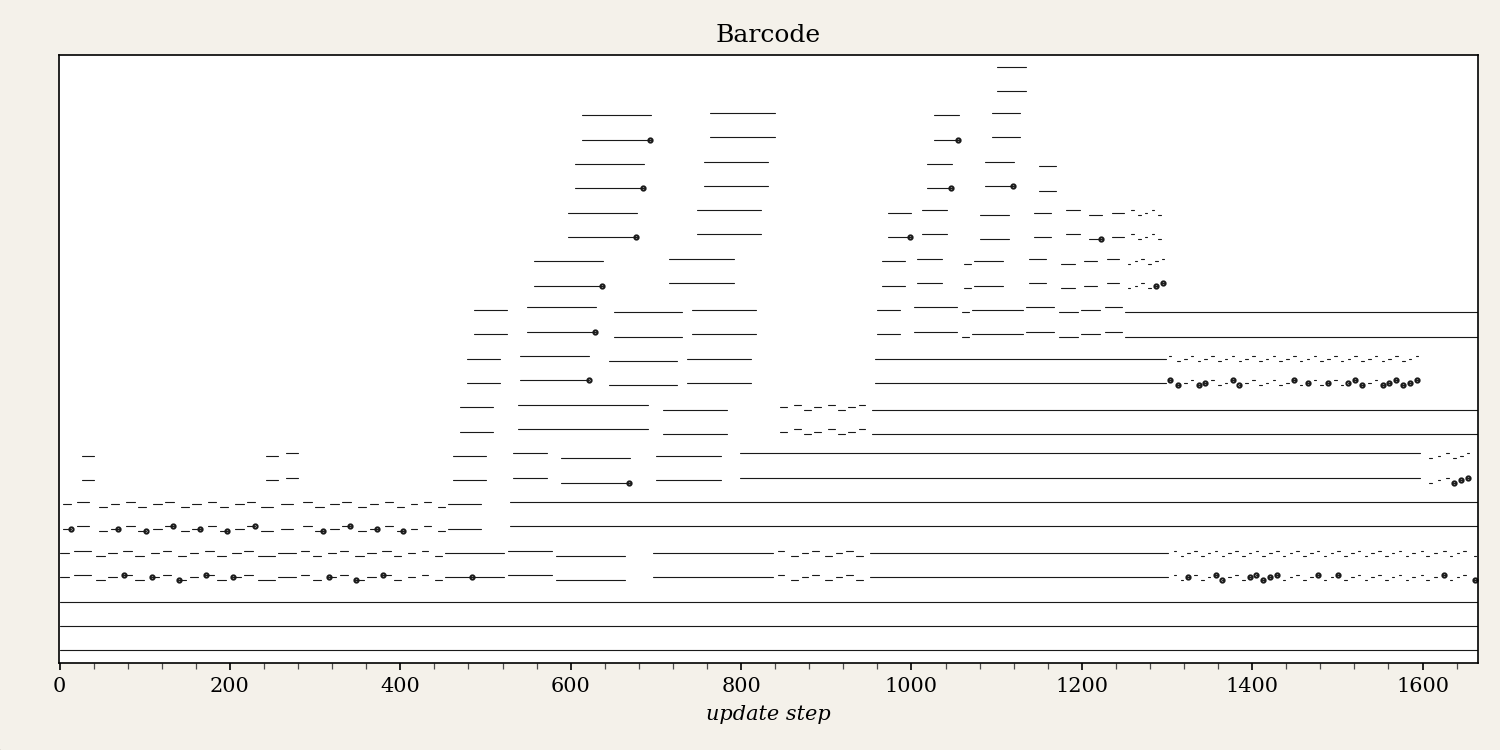}
    \caption{A simple barcode diagram. Bars persist for some number of update steps, each of these being an elementary operation. Each \textit{o} symbol marks a change in representative of a bar. Position on the vertical axis is entirely arbitrary.}
    \label{fig:large_barcode}
\end{figure}
}

\cancel{

The main property providing decomposability of the Conley-Morse persistence module is the attractor-repeller (AR) split theorem \cite[Theorem~5.2]{CMbarcodes2025}.
We adjust the result to the setting of this paper.

Let $\cV'\inscr\cV$ and 
	consider an isolating block $B$ in $\cV$ and its partition into blocks $B_1$ and $B_2$ in $\cV'$, that is $B=B_1\cup B_2$ and $B_1\cap B_2=\emptyset$. 
Moreover, assume that $\cl B_1 \cap B_2=\emptyset$.
Let us denote sets
    $N_0\coloneqq\mo B$, 
    $N_1\coloneqq B_1\cup N_0$, and
    $N_2\coloneqq \cl B$.
It is easy to check that 
    $B_1=N_1\setminus N_0$, $B_2=N_2\setminus N_1$, and $B=N_2\setminus N_0$.
In particular, $(N_1,N_0)$, $(N_2,N_1)$ and $(N_2,N_0)$ are index pairs for $B_1$, $B_2$ and $B$, respectively.
Together they form diagram~\eqref{eq:lefschetz_ar_split_inclusions}, 
    where $\imap{1}$, $\imap{2}$, $k$, $\lmap{1}$, and $\lmap{2}$ are inclusions, 
    and $\projmap{}$, $\projmap{1}$, $\projmap{2}$ are projections.
\begin{equation}\label{eq:lefschetz_ar_split_inclusions}
    \begin{tikzcd}
        & & B_2\arrow[lldd, hook', bend right, "\lmap{2}", swap]\\
        & & (N_2, N_1)\arrow[u,"\projmap{2}"]\\
        B & (N_2,N_0)\arrow[l,"\projmap{}"]\arrow[ru, hook, "\imap{2}"] &\\
        & & (N_1, N_0)\arrow[d,"\projmap{1}"]\arrow[lu, hook', "\imap{1}"]\arrow[uu, "k", hook, dashed] 
            & \\ 
        & & (B_1)\arrow[lluu, hook', bend left, "\lmap{1}"]
    \end{tikzcd}
\end{equation}
Map $\lmap{2}$, is the only map in diagram~\eqref{eq:lefschetz_ar_split_inclusions} which does not always induce a homomorphism in homology is $\lmap{2}$ in a direct way. 
This is due to the fact that $B_2$ is not necessarily closed in $B$; 
    therefore, a cycle in $\Cgroup(B_2)$ is not necessarily a cycle in $\Cgroup(B)$.

Maps $\imapx{1}$, $\imapx{2}$, and $k_\ast$ are homomorphisms induced by inclusions
    and $\projmapx{}$, $\projmapx{1}$, $\projmapx{2}$ are isomorphisms induced by projections given by Proposition~\ref{prop:lefschetz_relative_homology}.
We also write $g\coloneqq\lmapx{1}$ and 
    $f\coloneqq \projmapx{}\circ\imapx{2}^{-1}\circ\projmapx{2}^{-1}$, where $f$ is the homomorphism indirectly induced by $\lmap{2}$.
The maps are summarized in diagram~\eqref{eq:lefschetz_ar_split-appx}.
\begin{equation}\label{eq:lefschetz_ar_split-appx}
    \begin{tikzcd}
        & & H_d(B_2)\arrow[lldd, bend right, "f", swap]\\
        & & H_d(N_2, N_1)\arrow[u,"\projmapx{2}"]\\ 
        H_d(B) & H_d(N_2,N_0)\arrow[l,"\projmapx{}"]\arrow[ru, "\imapx{2}"] &\\
        & & H_d(N_1,N_0)\arrow[d,"\projmapx{1}"]\arrow[lu, "\imapx{1}"]\arrow[uu, "k_\ast", dashed]\\
        & & H_d(B_1)\arrow[lluu, bend left, "g"]
    \end{tikzcd}
\end{equation}

Triple $(N_0, N_1, N_2)$ induces a long exact sequence of a triple, and in consequence, $H_d(N_1,N_0)$, $H_d(N_2,N_1)$ and $H_d(N_2,N_0)$ split as presented in Diagram~\eqref{eq:lefschetz_ar_split-full};
    $\imapx{1}$ and $\imapx{2}$ also split 
    (see Proposition\ref{prop:B-to-B2-splits-into-0-f} or~\cite[Theorem~5.12]{CMbarcodes2025} for details).
Since $\projmapx{}$, $\projmapx{1}$, and $\projmapx{2}$ are isomorphisms $H_d(B_1)$, $H_d(B_2)$ and $H_d(B)$ split accordingly.
To simplify further discussion we write 
    $W\coloneqq H_d(B_1)$, $V\coloneqq H_d(B_2)$ and $X\coloneqq H_d(B)$, 
    $W'\coloneqq H_d(N_1,N_0)$, $V'\coloneqq H_d(N_2,N_1)$ and $X'\coloneqq H_d(N_2,N_0)$;
    moreover, we distinguish the split by writing $Y=Y_1\oplus Y_2$ for every $Y\in\{X,X',V,V', W,W'\}$.
We summarize these preparations in Diagram~\eqref{eq:lefschetz_ar_split-full}.
\begin{equation}\label{eq:lefschetz_ar_split-full}
    \begin{tikzcd}
        & & V_1\oplus V_2\arrow[lldd, bend right, "f=0\oplus\hat{f}", swap, sloped]
            & \projmapx{2}(V_1')\oplus \projmapx{2}(V_2') \ar[l, equal]  \\
        & & V_1'\oplus V_2'\arrow[u,"\projmapx{2}"]
            & \coker\imapx{2}\oplus \im\imapx{2} \ar[l,phantom, "\cong"]  \\
        X_1\oplus X_2 & X_1'\oplus X_2'\arrow[l,"\projmapx{}"]\arrow[ru, "\imapx{2}=0\oplus\hat{f}'", sloped] &\\
        & \im\imapx{1}\oplus\coker\imapx{1}\ar[u, phantom, "\cong", sloped] 
            & W_1'\oplus W_2'\arrow[d,"\projmapx{1}"]\arrow[lu, "\imapx{1}=\hat{g}'\oplus 0", sloped]\arrow[uu, "k_\ast", dashed]
            & \im \imapx{1}\oplus\ker\imapx{1} \ar[l,phantom, "\cong"]  \\
        & & W_1\oplus W_2\arrow[lluu, bend left, "g=\hat{g}\oplus 0", sloped]
            & \projmapx{1}(W_1')\oplus \projmapx{1}(W_2') \ar[l,equal]
    \end{tikzcd}
\end{equation}

\begin{proposition}\label{prop:B-to-B2-splits-into-0-f}
    We have $\imapx{1}\cong \hat{g}'\oplus 0$ and $\imapx{2}\cong 0\oplus \hat{f}'$, 
        where $\hat{g}':W_1'\rightarrow X_1'$ and $\hat{f}':X_2'\rightarrow V_2'$ are isomorphisms.  
\end{proposition}
\begin{proof}
	Long exact sequence of the triple $(N_0, N_1, N_2)$ implies $\im \imapx{1} = \ker \imapx{2}$, and therefore  
	\begin{equation*}
		X' \cong \im\imapx{1}\oplus\coker \imapx{1} 
			= \im\imapx{1}\oplus X'/\im \imapx{1} 
			= \im\imapx{1}\oplus X'/\ker \imapx{2}
			\cong \im\imapx{1}\oplus\im \imapx{2}.
	\end{equation*}
	Since $X_2'\cong\coker \imapx{1}$ and $V_2=\im \imapx{2}$, the above equation implies the isomorphism 
        $\hat{f}': X_2\rightarrow V_2$.
	Since $V_1'\cong\coker \imapx{2}$, we get $\imapx{2} \cong 0\oplus \hat{f}'$.
    The split of $\imapx{1}$ is straightforward.
    \qed
\end{proof}

The following result is straightforward, because $\projmapx{2}$ and $\projmapx{}$ are isomorphisms and the split of $\imapx{1}$ given by Proposition~\ref{prop:B-to-B2-splits-into-0-f}.

\begin{proposition}\label{prop:f-well-defined-apdx}
Homomorphism $f$ is well defined.
    Moreover, it splits into $f\cong 0\oplus \hat{f}$, where $\hat{f}:V_2\rightarrow X_2$ is an isomorphism.
\end{proposition}



With the below lemma we show that $f$ can be well characterized also at the chain level.
We write $\partial_B$ for the boundary operator with respect to Lefschetz complex $B$, 
    and $[c]_Y$ to denote homology group of chain $c$ within $Y\in\{X,V,W\}$.

\begin{lemma}\label{lem:the-reversed-homomorphism-apdx}
    Let $c\in\Zgroup(B_2)$ such that $c\neq 0$.
    The following statements hold:
    \begin{enumerate}[label=\arabic*)]
        \item\label{it:the-reversed-homomorphism-chain-exists-apdx}
            $[c]_V\in V_2$ (or $[c]_V\in\coim f$) if and only if there exists an $e\in\Cgroup(B_1)$ such that no subchain of $e$ is in $\Zgroup(B_1)$ and  $d\coloneqq c+e\in\Zgroup(B)$ and $[d]_X\in X_2$; in particular $f([c]_{V})= [d]_X$,
            \michal{no subcycle can also be skipped here}
            we call $e$ a \emph{complementary chain for $c$ with respect to $f$},
        \item\label{it:the-reversed-homomorphism-chain-e-apdx} 
            for any $e\in\Cgroup(B_1)$ such that $d\coloneqq c+e\in\Zgroup(B)$ and $[d]_X\in X_2$, 
            we have $f([c]_{V})= [d]_X$;
            \michal{the other version, that is $[d]_X$ is a basis element is not completely true, the basis cannot be arbitrary, but consistent with the split into $X_1$ and $X_2$}
            \michal{we can skip the subcycle assumption}
        \item\label{it:the-reversed-homomorphism-non-zero-apdx}
            $f([c]_{V})\not = 0$ $\Leftrightarrow$
            there exists $e\in\Cgroup(B_1)$ such that $d\coloneqq c+e\in\Zgroup(B)$
            $\Leftrightarrow$ $[\partial_B c]_W=0$
            \michal{the first implication $\Rightarrow$ is true only if $[c]_V\in V_2$}
    \end{enumerate}
\end{lemma}

\begin{proof}
    Throughout the proof we write $\overline{c}\coloneqq c+\Cgroup(N_1)\in\Zgroup(N_2,N_1)$ and $\overline{d}\coloneqq c+e+\Cgroup(N_0)\in\Zgroup(N_2,N_0)$.
    In particular, $\projmap{2}^{-1}(c) = \overline{c}$ and $\projmap{}(\overline{d})=d$.

    We first show existence of a chain $e$ mentioned in \ref{it:the-reversed-homomorphism-chain-exists-apdx}.
    Assume $[c]_V\in V_2$, then $[\overline{c}]\in\im\imapx{2}=V_2'$.
    Since $\imapx{2}$ is induced by inclusion there exists $e\in\Cgroup(B_1)$ such that for
        $\overline{d}\coloneqq c+e + \Cgroup(N_0)\in\Zgroup(N_2,N_0)$ we have $\imapx{2}([\overline{d}])=[\overline{c}]$.
    Clearly, $d\in\Zgroup(B)$.
    If $e$ contains a subchain $b$ which is a cycle in $B_1$ we can simply take $e'=e+b$ instead, because $d'\coloneqq c+e'\in\Zgroup(B)$ and $\imapx{2}([\overline{d'}])=[\overline{c}]_V$.
    We can continue purging all subcycles. 
    Assume that $e$ has no subcycles.
    We finally show that $[d]_X\in X_2$.
    If it is not true, we can decompose it into $d=c+e+a+a+b$, where $a$ is a subchain of $e$, $a+b\in\Zgroup(B_1)$ and $[c+e+a]_X\in X_2$ and $[a+b]\in X_1$.
    In other words, $[a+b]_X$ corresponds for the $X_1$ component of $[d]_X$.
    Thus, we can take $e'\coloneqq e+a$ instead of $e$. 
    Since $a$ is a subchain of $e$, we only remove elements, and therefore no extra subcycle can be introduced. 
    The opposite inclusion in \ref{it:the-reversed-homomorphism-chain-exists-apdx} is straightforward.

    To show~\ref{it:the-reversed-homomorphism-chain-e-apdx} consider $e,e'\in\Cgroup(B_1)$ such that $d\coloneqq c+e$ and $d'\coloneqq c+e'$ are in $\Zgroup(B)$ and $[d]_X,[d']_X\in X_2$. 
    Clearly, $\imap{2}(\overline{d})=\imap{2}(\overline{d'})=c$
        and therefore, $\imapx{2}([\overline{d}]_{X'})=\imapx{2}([\overline{d'}]_{X'})=[\overline{c}]_{V'}$.
    If $[c]_V\neq 0$, then by Proposition~\ref{prop:B-to-B2-splits-into-0-f} we have $[\overline{c}]_{V'}\in V_2'$ and necessarily $[\overline{d}]_{X'}=[\overline{d'}]_{X'}$.
    If $[c]_V= 0$   
    
    To show~\ref{it:the-reversed-homomorphism-non-zero-apdx} notice that  $f([c]_V)\neq 0$ implies existence of $e$ by \ref{it:the-reversed-homomorphism-chain-exists-apdx}. 
    Since $\partial_B c, \partial_B e\in\Cgroup(B_1)$, we have $0=\partial_B d= \partial_B c + \partial_B e$, which implies $\partial_B c = \partial_B e$. 
    Thus, $\partial_B c\in\Bgroup(B_1)$, and therefore $[\partial_B c]_W=0$.
    
    Finally assume that $[\partial_B c]_W=0$.
    If follows that $a\coloneqq\partial_B c\in\Bgroup(B_1)$ and that there exists $e\in\Cgroup(B_1)$ such that $a=\partial_B(e)$.
    Therefore, we have $d\coloneqq c+e\in\Zgroup(B)$.
    Since, $\imap{2}(d)=c$ and $0\neq[\overline{c}]_{V'}]\in V_2'$, we get $f([c]_V)\neq 0$ from Proposition~\ref{prop:f-well-defined-apdx}.
    \qed
\end{proof}

\begin{corollary}
    If $B_1$ or $B_2$ is a singleton then \ref{lem:the-reversed-homomorphism-apdx} we can simplify the assumptions to $e\in\Cgroup(B)$ instead of $e\in\Cgroup(B_1)$ in each statement.
\end{corollary}

\begin{corollary}\label{cor:splitting-into-im-g-im-f-apdx}
    $X=X_1\oplus X_2=\im g\oplus\im f$;
    $V=V_1\oplus V_2=\ker f\oplus\coim f$;
    $W=W_1\oplus W_2=\coim g\oplus\ker g$.
\end{corollary}

\hspace{0.5cm}
The following remark is a standard result.
\begin{proposition}\label{prop:chains-in-map-g}
    Let $c\in\Zgroup(B_1)$.
    Then $g([c]_{W})=0$ if and only if there exists $d\in\Cgroup(B)$ such that $\partial_B d = c$.
\end{proposition}
}

\end{document}